\documentclass[11pt]{amsart}
\usepackage[T1]{fontenc}
\usepackage{lmodern}
\usepackage{amsmath,amssymb,amsthm,mathtools}
\usepackage[margin=1in]{geometry}
\usepackage{microtype}
\usepackage[hidelinks]{hyperref}
\numberwithin{equation}{section}
\newtheorem{theorem}{Theorem}[section]
\newtheorem{lemma}[theorem]{Lemma}
\newtheorem{proposition}[theorem]{Proposition}
\newtheorem{corollary}[theorem]{Corollary}
\theoremstyle{remark}

\newcommand{\E}{\mathbb E}
\newcommand{\Var}{\operatorname{Var}}
\newcommand{\Cov}{\operatorname{Cov}}
\newcommand{\Li}{\operatorname{Li}_2}
\newcommand{\D}{\mathbb D}

\title[Gaussian zero-count fluctuations in annuli]
{Gaussian zero-count fluctuations in annuli}
\author{Xiaolin Zhu}
\thanks{Xiaolin Zhu: School of Mathematics and Statistics, Taiyuan Normal University,
Taiyuan 030012, P. R. China. E-mail: \texttt{xiaolinzhu@tynu.edu.cn}.
ORCID: \url{https://orcid.org/0009-0004-0745-896X}.}
\author{Feng Guo}
\thanks{Corresponding author: Feng Guo, School of Mathematics,
South China University of Technology, Guangzhou 510640, P. R. China.
E-mail: \texttt{70207994@nuaa.edu.cn}.}
\date{}
\hypersetup{
 pdftitle={Gaussian zero-count fluctuations in annuli},
 pdfauthor={Xiaolin Zhu and Feng Guo},
 pdfkeywords={Gaussian analytic functions, annular zero counts, central limit theorem,
 boundary fluctuations, thin annuli, Wiener chaos}}

\begin{document}

\begin{abstract}
We determine joint zero-count fluctuations of the self-dual weighted
Szeg\H{o} Gaussian analytic function in fixed and thinning annuli.
For a fixed annulus, the counts near its two boundary components
converge jointly to independent standard normal variables after
centering and marginal normalization, at arbitrary relative rates of
boundary approach. This answers the simultaneous central limit
question of Katori and Shirai for this model. The unscaled covariance
converges to a negative constant, which we evaluate in complete elliptic
integrals. A sharp variance asymptotic for the difference between an annular
winding number and its coupled disk count yields quantitative normal
approximation bounds.
For thinning annuli with inner radius $q=e^{-h}$, the counts in every
fixed finite family of interior transverse bands converge jointly to
a Gaussian vector after centering and multiplication by $\sqrt h$. We give its spectral
covariance matrix, prove strict negative covariance for disjoint bands,
and identify its nullspace from the interval indicators. The
fixed-annulus proof couples the Laurent series to two independent
Gaussian power series. For thinning annuli, we combine exact
periodization of a strip kernel, a uniform Laguerre-chaos tail bound,
and connected Gaussian diagrams.
\end{abstract}

\maketitle

\section{Introduction}\label{sec:introduction}

For $q\in(0,1)$, consider the Gaussian Laurent series
\[
 F_q(z)=\sum_{n\in\mathbb Z}
       \frac{\xi_n z^n}{\sqrt{1+q^{2n+1}}},
       \qquad q<|z|<1,
\]
where $(\xi_n)_{n\in\mathbb Z}$ are independent standard proper complex
Gaussian variables. This is the self-dual member, with weight parameter
$r=q$, of the weighted Szeg\H{o} family of Katori and Shirai
\cite{KS}. They identify its zeros as a permanental-determinantal point
process and prove the inversion symmetry exchanging the two boundary
components. We study the joint fluctuations of zero counts near these
components and in annuli whose width tends to zero.

For the disk Gaussian power series, Peres and Vir\'ag
\cite[Theorem~2(i) and Corollary~3(iii)]{PV} express the count in a
centered disk as a sum of independent Bernoulli variables and obtain
its central limit theorem. For the annular model above, Katori and
Shirai ask in their 2022 paper \cite[Section~4, item~(7)]{KS} for a
useful representation and simultaneous central limit theorems for the
counts in $\{s<|z|<\sqrt q\}$ and $\{\sqrt q<|z|<t\}$ as
$s\downarrow q$ and $t\uparrow1$. The field in their question is
the self-dual field $F_q$.

Theorem~\ref{thm:joint} answers the simultaneous central limit part of
this question. The two counts converge jointly to independent standard
normal variables after centering and marginal normalization, with no restriction on
their relative rates of boundary approach. The proof approximates the standardized annular counts in $L^2$
by standardized counts from two independent disk fields. The finite
covariance corrections describe information lost under normalization:
the unscaled covariance converges to an explicitly evaluated negative
constant even though both marginal variances diverge. Theorem~\ref{thm:sharp}
quantifies the error in the coupling and yields normal approximation
bounds.

Our second limit concerns thinning annuli, $q=e^{-h}$ with
$h\downarrow0$, and a fixed finite family of bands at interior fractions
of the annular width. Theorem~\ref{thin:thm:bands} gives their joint
Gaussian limit with an explicit spectral covariance matrix. Here the
covariance of disjoint bands persists at the fluctuation scale and is
strictly negative. The same formula treats overlapping bands and
identifies all degenerate directions through linear relations among
their interval indicators.

For fixed $q$, write $F=F_q$ and $\D=\{z:|z|<1\}$.
Let $Z_f(B)$ denote the number of zeros of $f$ in a set $B$, with multiplicity.
Set
\[
 m(r)=\frac{r^2}{1-r^2},\qquad v(r)=\frac{r^2}{1-r^4}.
\]

\begin{theorem}[Joint boundary central limit theorem]\label{thm:joint}
Let $q<s<\sqrt q<t<1$ and put
\[
 N_-(s)=Z_F\{s<|z|<\sqrt q\},\qquad
 N_+(t)=Z_F\{\sqrt q<|z|<t\}.
\]
For every deterministic pair of sequences $s\downarrow q$ and $t\uparrow1$,
with no restriction on their relative rates,
\[
 \left(
 \frac{N_-(s)-m(q/s)}{\sqrt{v(q/s)}},
 \frac{N_+(t)-m(t)}{\sqrt{v(t)}}
 \right)\ \Longrightarrow\ \mathcal N(0,I_2).
\]
The same conclusion holds when each count is centered by its exact
expectation and divided by its exact marginal standard deviation.
In particular,
\[
 \E N_-(s)=m(q/s)+O_q(1),\quad \Var N_-(s)\sim v(q/s),
\]
\[
 \E N_+(t)=m(t)+O_q(1),\quad \Var N_+(t)\sim v(t).
\]
\end{theorem}

Propositions~\ref{prop:finite} and~\ref{prop:elliptic}
compute the limiting covariance $\kappa_q<0$ and the corresponding
finite marginal variance corrections. To measure the approximation
used in the proof, we couple $F$ to disk fields on the same coefficient
space.

For an analytic function without zeros on the circle $|z|=r$,
its winding number is
\[
 W_f(r)=\frac1{2\pi}\int_0^{2\pi}
 \Re\frac{re^{i\theta}f'(re^{i\theta})}{f(re^{i\theta})}\,d\theta.
\]
For the Gaussian fields considered here, each fixed circle in the
domain is almost surely free of zeros.
On the same probability space, define the independent disk
Gaussian analytic functions (GAFs)
\[
 G_+(z)=\sum_{n\ge0}\xi_nz^n,\qquad
 G_-(z)=\sum_{n\ge0}\xi_{-n-1}z^n.
\]
Write $M_\pm(r)=W_{G_\pm}(r)=Z_{G_\pm}(r\D)$.
\begin{theorem}\label{thm:sharp}
Put $c_n=(1+q^{2n+1})^{-1/2}$ and
\[
 \beta_q=\sum_{n\ge0}(c_n-1)^2+\sum_{n<0}c_n^2
 =2\sum_{n\ge0}\left(1-\frac1{\sqrt{1+q^{2n+1}}}\right)>0.
\]
The sums converge absolutely. If $\delta=1-r^2\downarrow0$, then
\[
 \Var\bigl(W_F(r)-M_+(r)\bigr)\sim\sqrt{\beta_q}\,\delta^{-1/2}.
\]
More precisely,
\[
 \Var W_F(r)-\Var M_+(r)=O_q(1),
\]
\[
 \Cov(W_F(r),M_+(r))-
 \Var M_+(r)\sim-\frac{\sqrt{\beta_q}}2\,\delta^{-1/2}.
\]
\end{theorem}

The coupling-error variance has order $(1-r^2)^{-1/2}$, whereas the
disk-count variance has order $(1-r^2)^{-1}$. Thus the approximation
is negligible at the marginal fluctuation scale. The sharper estimate
also gives the Wasserstein and rectangle bounds in
Proposition~\ref{prop:rates}.

The fixed-$q$ comparison does not give estimates uniform as $q\uparrow1$.
We therefore treat thinning annuli separately.
Let $h\downarrow0$, put
$q_h=e^{-h}$, and consider
\begin{equation*}
 F_h(z)=\sum_{n\in\mathbb Z}
       \frac{\xi_nz^n}{\sqrt{1+e^{-h(2n+1)}}},
       \qquad e^{-h}<|z|<1,
\end{equation*}
where $(\xi_n)_{n\in\mathbb Z}$ are independent standard proper complex Gaussians.
A radius is written as $|z|=e^{-hy}$, with $0<y<1$.
Fix a positive integer $r$ and nonempty intervals $I_i=(d_i,e_i)$,
$0<d_i<e_i<1$, for $1\le i\le r$. They may overlap or coincide.
Define the corresponding zero counts and spectral probabilities by
\begin{equation*}
 \begin{gathered}
 N_{i,h}=Z_{F_h}\{e^{-he_i}<|z|<e^{-hd_i}\},\qquad
 D(k)=\int_0^1e^{-2ky}\,dy,\\
 \lambda_I(k)=D(k)^{-1}\int_Ie^{-2ky}\,dy,\qquad
 \lambda_\varnothing(k)=0.
 \end{gathered}
\end{equation*}
For $I=(d,e)$, the spectral probability equals
$(e^{-2kd}-e^{-2ke})/(1-e^{-2k})$ for $k\ne0$, and $e-d$ at $k=0$.
It decays exponentially at both ends of the real line whenever
$0<d<e<1$. Intersections of intervals are interpreted up to null sets.

\begin{theorem}\label{thin:thm:bands}
The matrix
\begin{equation}\label{thin:eq:band-sigma}
 \Sigma_{ij}=\int_{\mathbb R}
    \{\lambda_{I_i\cap I_j}(k)-\lambda_{I_i}(k)\lambda_{I_j}(k)\}\,dk
\end{equation}
has finite entries, is positive semidefinite, and is the limiting
covariance matrix of the counts:
\begin{equation}\label{thin:eq:band-clt}
 h\Cov(N_{i,h},N_{j,h})\longrightarrow\Sigma_{ij},\qquad
 \sqrt h\,(N_{i,h}-\E N_{i,h})_{i=1}^r
                \Longrightarrow\mathcal N_r(0,\Sigma).
\end{equation}
Moreover,
\begin{equation}\label{thin:eq:band-means}
 \E N_{i,h}=\frac\pi{2h}\{\cot(\pi d_i)-\cot(\pi e_i)\}
                    +O(h^{-1}e^{-\pi^2/h}),
\end{equation}
so the leading terms in \eqref{thin:eq:band-means} may replace the exact
expectations in \eqref{thin:eq:band-clt}.
The nullspace of $\Sigma$ consists exactly of the vectors $v\in\mathbb R^r$
such that $\sum_i v_i\mathbf1_{I_i}=0$ almost everywhere. Thus $\Sigma$
is positive definite precisely when the interval indicators are linearly
independent. Every limiting marginal variance is positive, and disjoint
nonempty bands have strictly negative limiting covariance.
\end{theorem}

The number of bands and all their endpoints are fixed in this theorem.
For disjoint nonempty bands, the standardized limiting coordinates
have negative correlation, whereas the fixed-annulus limit has
independent coordinates.

Sodin and Tsirelson \cite[Sections~2.2--2.4]{ST} use polynomial
approximation and Gaussian pairing diagrams to prove asymptotic
normality for smooth zero statistics. Buckley and Nishry
\cite[Proposition~17]{BN} give an $L^2$ chaos expansion of disk zero
counts and study their limiting laws. Feldheim
\cite{FeldheimLLN,FeldheimVar} studies limiting horizontal zero measures
and variance growth for stationary strip functions. Buckley and
Feldheim \cite{BF} prove winding central limit theorems for fixed
stationary Gaussian processes on the real line.

For the upper-half-plane Hardy--Szeg\H{o} zero process, Ai, Guo and Zhou
\cite[Theorem~1.6]{AGZ} prove a functional central limit theorem for
long horizontal windows with a fixed admissible height window.
For windows between two fixed positive heights, Ai, Fang and Hou
\cite[Theorems~1.2 and~1.4]{AFH} prove a large-deviation principle and
a strong Szeg\H{o} expansion for the log-moment generating function.
Our thinning limit varies the annular field itself: in logarithmic
coordinates the period and the length of observation both equal
$2\pi/h$, and the counts run around the entire cylinder in fixed
transverse bands. Passing from the limiting strip field to these
annular counts requires covariance and chaos-tail estimates uniform
in $h$.

\subsection*{Outline of the proofs}

For the fixed-annulus argument, keep $q$ fixed and write $F=F_q$.
The comparison is made at the level of centered winding numbers. The
argument principle writes each annular count as a difference of two
windings, one near the relevant boundary and one on the fixed circle
$|z|=\sqrt q$. The latter has finite second moment by
Lemma~\ref{lem:covariance}. For the outer boundary, it therefore suffices
to prove
\[
 \bigl\|(W_F(r)-\E W_F(r))-(M_+(r)-m(r))\bigr\|_2
                  =o(\sqrt{v(r)}),\qquad r\uparrow1.
\]
The coupling uses the same Gaussian coefficients with nonnegative
indices in $F$ and $G_+$. Their covariance kernels have a common singular part. With $u=z\bar w$, write
\[
 \begin{aligned}
 K_F(u)&=\E F(z)\overline{F(w)}=\frac1{1-u}+B_F(u),\\
 K_G(u)&=\E G_+(z)\overline{G_+(w)}=\frac1{1-u},\\
 K_c(u)&=\E F(z)\overline{G_+(w)}=\frac1{1-u}+B_c(u).
 \end{aligned}
\]
The deterministic weights of the nonnegative Laurent modes approach
their disk counterparts geometrically, while the weights of the negative
Laurent modes decay geometrically as the absolute index increases. Consequently $B_F$ and $B_c$ are analytic on a
neighborhood of the unit circle, with bounded derivatives there for
fixed $q$. To turn this kernel comparison into a winding estimate, we
must control the singularity created by logarithmic differentiation.

For nondegenerate centered jointly proper Gaussian fields $f,g$, put
$K_{fg}(z,w)=\E f(z)\overline{g(w)}$ and
\[
 S_{fg}(z,w)=\frac{|K_{fg}(z,w)|^2}
                   {K_{ff}(z,z)K_{gg}(w,w)}.
\]
The Gaussian log-modulus identity gives
\[
 \Cov(\log|f(z)|,\log|g(w)|)=\tfrac14\Li(S_{fg}(z,w)),
 \qquad \Li(s)=\sum_{m\ge1}\frac{s^m}{m^2}.
\]
For the three pairs $(F,F)$, $(F,G_+)$ and $(G_+,G_+)$, joint rotational
invariance reduces the double circular average to a single angular
integral. Differentiation in the two logarithmic radii then recovers
the winding covariance. This differentiation is
justified in $L^2$ by annular Jensen secant bounds; the angular
singularity is integrable. In squared radial variables $x,y$, with
$S=S_{fg}(\sqrt x e^{i\theta},\sqrt y)$, this gives
\[
 \Cov(W_f(\sqrt x),W_g(\sqrt y))
 =\frac{xy}{2\pi}\int_{-\pi}^{\pi}
                         \partial_x\partial_y\Li(S)\,d\theta.
\]
Where $0<S<1$, put $L=\log S$; then
\[
 \partial_x\partial_y\Li(S)
 =\frac{S}{1-S}L_xL_y-\log(1-S)L_{xy}.
\]
At kernel zeros, the formula in terms of derivatives of $S$ in
Lemma~\ref{lem:covariance} gives the continuous extension.
The winding comparison uses bounds on the correlation gap $1-S$ and
the radial derivatives in addition to the analytic kernel estimates.

Set $a=r^2$, $\delta=1-a$, and let $S_F,S_c$ be the squared normalized
self-correlation of $F$ and cross-correlation of $(F,G_+)$ on the circle
of radius $r$. Their disk counterpart is
\[
 S_0(a,\theta)=\frac{\delta^2}
                       {\delta^2+2a(1-\cos\theta)}.
\]
Lemma~\ref{lem:gaps} proves
\[
 1-S_F\ge c_q(1-S_0),\qquad
 1-S_c\ge c_q\{\delta+1-S_0\}.
\]
For the first bound, the normalized Laurent coefficient masses dominate
a fixed multiple of the disk's geometric masses. For the second, the
negative Laurent modes supply variance independent of $G_+$, producing
the additional gap of order $\delta$. These bounds control both terms
of the differentiated dilogarithm near coincidence.
Under $\theta=\delta u$, the analytic corrections disappear from the
leading profile of all three winding covariances. After multiplication
by $\delta^2$, the rescaled integrands are bounded by
$C_q(1+|\log|u||)$ near zero and by
$C_qu^{-2}$ for $|u|\ge1$, so the same integrable profile determines
their limits. Lemma~\ref{lem:asymptotic} evaluates its integral and gives
\[
 \Var W_F(r)\sim\Cov(W_F(r),M_+(r))\sim v(r)
                         \sim\frac1{2\delta}.
\]
The variance identity
\[
 \Var(W_F(r)-M_+(r))
 =\Var W_F(r)+v(r)-2\Cov(W_F(r),M_+(r))
                         =o(\delta^{-1})
\]
is precisely the centered $L^2$ comparison needed above. The diagonal
kernel decomposition also gives the required mean comparison:
\[
 \E W_F(r)=a\frac{K_F'(a)}{K_F(a)}=m(r)+O_q(1).
\]

The inner boundary uses the same estimate after the exact inversion
\[
 \widetilde F(w)=\frac{\sqrt q}{w}F(q/w),\qquad
 W_{\widetilde F}(r)=-1-W_F(q/r).
\]
Reversing the Laurent indices shows that $\widetilde F$ has the same
law as $F$, and that its nonnegative modes use the coefficients
$\xi_{-n-1}$ with the original weights. Its disk comparison is therefore
$G_-$. The argument principle gives
\[
 \begin{aligned}
 N_+(t)&=W_F(t)-W_F(\sqrt q),\\
 N_-(s)&=W_{\widetilde F}(q/s)-W_{\widetilde F}(\sqrt q).
 \end{aligned}
\]
The fixed-circle windings and the bounded mean differences vanish
after division by the respective diverging standard deviations. Hence
each centered and normalized annular count differs in $L^2$ from the
correspondingly centered and normalized disk count by a quantity tending
to zero. Peres and Vir\'ag's
independent Bernoulli representation
\cite[Theorem~2(i)]{PV} gives the scalar disk central limit theorem.
The two disk counts are independent because they use disjoint Gaussian
coefficient blocks. This proves the joint limit by Slutsky's theorem.
Each error is controlled at its own marginal scale, which permits
arbitrary relative rates of boundary approach.

The finite covariance correction requires retaining the terms that
disappear in this normalization. For the Laurent weights
$c_n=(1+q^{2n+1})^{-1/2}$, self-duality gives
$c_{-n-1}^2=1-c_n^2$ and the stronger cancellation $B_F(1)=0$.
After subtracting the disk variance integrand, this cancellation leaves
an integrable residual, while differentiation of the reciprocal
diagonal kernel gives the covariance limits between an approaching
boundary circle and the fixed middle circle.
Combining these limits in the two winding differences gives
$\Cov(N_-(s),N_+(t))\to\kappa_q$ and the finite marginal corrections
in Proposition~\ref{prop:finite}. To evaluate $\kappa_q$, we pair the
Laurent modes on $|z|=\sqrt q$ and use the Jacobi cosine series:
\[
 \frac{K_F(qe^{i\theta})}{K_F(q)}
 =e^{-i\theta/2}\operatorname{cn}(\mathbf K\theta/\pi,k).
\]
Here $k$ is the elliptic modulus with nome $q$, and $\mathbf K$ is its
complete elliptic integral of the first kind; the definitions are given
in Proposition~\ref{prop:elliptic}. The phase and inversion symmetry
make the first radial derivatives of the squared normalized
self-correlation vanish on that circle.
An integration by parts then expresses the middle-circle winding
variance in complete elliptic integrals; Parseval's identity and
differentiation with respect to the nome evaluate the other terms in
$\kappa_q$. Proposition~\ref{prop:elliptic} gives the resulting formula
and its strict negative sign. Theorem~\ref{thm:sharp} measures the
remaining discrepancy between the coupled windings.

To obtain the sharp error, we retain the first correlation defect at
coincident points. With $a=r^2=1-\delta$, put
\[
 \kappa=\frac{K_c(a)^2}{K_F(a)K_G(a)}.
\]
The coefficient discrepancy in $\beta_q$ gives
$1-\kappa=\beta_q\delta+O_q(\delta^2)$, while the analytic corrections
with real coefficients give
$S_c(a,\theta)=\kappa S_0(a,\theta)\exp(O_q(\theta^2))$.
In the winding covariance formula, the differences of the rational
terms have bounded angular integrals. The leading cross-covariance
correction comes from the logarithmic term, through
\[
 \log\frac{1-S_c}{1-S_0}
 =\log\left(1+
      \frac{(1-\kappa)\delta^2}{2a(1-\cos\theta)}\right)
      +O_q(\delta^2)
\]
on a fixed small punctured angular neighborhood of zero. There the
defect of order $\delta$ balances the disk gap
$1-S_0\asymp\theta^2/\delta^2$ for $|\theta|\le\delta$ at the smaller
scale $\theta=\delta^{3/2}v$. On this scale, the logarithm tends to
$\log(1+\beta_q/v^2)$, and the mixed radial derivative multiplying it
has size $\delta^{-2}$. The angular width $\delta^{3/2}$ therefore
produces a correction of order $\delta^{-1/2}$. The integrable bound
$C_q\log(1+C_q/v^2)$ justifies the rescaling, and
\[
 \int_{\mathbb R}\log(1+\beta_q/v^2)\,dv
                              =2\pi\sqrt{\beta_q}
\]
determines its coefficient. Equation~\eqref{eq:covdifference} then gives
the cross-covariance asymptotic in Theorem~\ref{thm:sharp}; the variance
identity above gives the sharp coupling variance.

For the quantitative limit, let $r_-=q/s$, $r_+=t$, and
$\eta_i=1-r_i^2$ for $i\in\{-,+\}$. The coupling error has standard
deviation $O_q(\eta_i^{-1/4})$, whereas the count standard deviation
has order $\eta_i^{-1/2}$. Including the fixed-circle winding and the
change to exact marginal normalization therefore gives an $L^2$
error $O_q(\eta_i^{1/4})$ between the standardized annular and disk
counts; see \eqref{eq:rate-transfer}. The Bernoulli disk law has
Wasserstein error $O(\eta_i^{1/2})$ by the scalar Stein bound proved in
Section~\ref{sec:fixedq-rates}. Coordinatewise replacement of the
independent disk laws and the vector coupling thus give the
Wasserstein bound in Proposition~\ref{prop:rates}. Smoothing the scalar
disk Wasserstein bound gives Kolmogorov error $O(\eta_i^{1/4})$;
independence gives a disk-pair rectangle error of the corresponding
summed order. To transfer this bound, a boundary enlargement of width
$\varepsilon_i$ costs
$O(\varepsilon_i)$ under the Gaussian law, while Chebyshev's
inequality bounds the coupling failure by
$O_q(\eta_i^{1/2}\varepsilon_i^{-2})$.
Choosing $\varepsilon_i=\eta_i^{1/6}$ gives the stated rectangle bound,
with the disk-pair error of smaller order.

For the thinning limit, use logarithmic coordinates $u=y+ix$ and put
$T=2\pi/h$. The transformation $z=e^{-hu}$ turns a full turn around
the annulus into a longitudinal interval of length $T$. Define
\[
 Y_h(u)=\sqrt h\,e^{-hu/2}F_h(e^{-hu}),\qquad
 J(w)=\frac\pi{2\sin(\pi w/2)},\quad 0<\Re w<2.
\]
The analytic multiplier is nonvanishing and hence preserves the zero
set. Sampling at the half-integer frequencies $h(n+1/2)$ gives the
exact covariance formula
\[
 \E Y_h(u)\overline{Y_h(v)}=A_h(u+\bar v),\qquad
 A_h(w)=\sum_{\nu\in\mathbb Z}(-1)^\nu J(w+i\nu T).
\]
The field $Y_h$ is antiperiodic in $x$, while its log modulus and the
even Gaussian polynomials used below are periodic. The shifted copies
of $J$ give both convergence to the strip kernel on bounded sets and
control over the growing integration interval. More precisely, with
$d_T(t)=\min_{\nu\in\mathbb Z}|t-\nu T|$, every fixed derivative order
satisfies
\[
 |A_h^{(j)}(y+z+it)|\le C_j e^{-\pi d_T(t)/2}.
\]
Here and below the transverse levels lie in a fixed compact interval
inside $(0,1)$, and the constants are independent of small $h$.

Let $W_h(y)=W_{F_h}(e^{-hy})$ and
\[
 S_h(y,z;t)=\frac{|A_h(y+z+it)|^2}{A_h(2y)A_h(2z)}.
\]
The winding operator on the circular logarithmic average is
$-h^{-1}\partial_y$. Since $hT=2\pi$, the same log-modulus identity
as in the fixed-annulus comparison gives
\[
 h\Cov(W_h(y),W_h(z))
 =\frac1{8\pi}\int_{-T/2}^{T/2}
                       \partial_y\partial_z\Li(S_h(y,z;t))\,dt
 \longrightarrow L(y,z),
\]
where $L$ is obtained by replacing $A_h$ with $J$ in the normalized
correlation and integrating over $\mathbb R$.
The periodization bound controls the tails of this integral.
Near coincidence,
$1-S_h\asymp(y-z)^2+t^2$ and the first radial derivatives of
$\log S_h$ are $O(|y-z|+t^2)$, while
$\partial_y\partial_z\log S_h=O(1)$. Inserting these bounds into the
differentiated dilogarithm gives the uniform integrable majorant
$C(1+|\log|t||)$. Thus the covariance limit also holds at $y=z$,
and $L$ is continuous. These bounds supply the radial and integral
control needed for endpoint winding differences.

We identify this limiting covariance by a second computation in the
strip. The map
$\phi(u)=(e^{i\pi u}-i)/(e^{i\pi u}+i)$ maps $0<\Re u<1$ onto the
disk. The pullback of $G_+$, multiplied by
$\sqrt{2\pi}e^{i\pi u/2}/(e^{i\pi u}+i)$, has covariance
$J(u+\bar v)$. Equality of the proper Gaussian laws and the nonvanishing
multiplier transfer Peres--Vir\'ag's determinantal zero correlations
\cite[Theorem~1]{PV} to this strip field. Its zero kernel relative to
Lebesgue area has the representation
\[
 B(y+ix,z+ix')=\frac1{2\pi}\int_{\mathbb R}
                   p_k(y)p_k(z)e^{-ik(x-x')}\,dk,
 \qquad p_k(y)=\frac{e^{-ky}}{\sqrt{D(k)}}.
\]
The normalization $\int_0^1p_k(y)^2\,dy=1$ makes
$d\nu_k(y)=p_k(y)^2\,dy$ a probability measure and gives
$\nu_k(I)=\lambda_I(k)$.

To relate this spectral formula to annular windings, we compute the
variance per unit longitudinal length of a smooth strip zero statistic
with transverse weight $g$ in two ways: from the determinantal zero
correlations and from the log-modulus covariance. Plancherel's identity in the first computation
produces $\operatorname{Var}_{\nu_k}(g)$, and integration by parts in
the second produces $L$. Smoothing a compactly supported step function
$g$ then gives the endpoint identity
\[
 g'=\sum_j\gamma_j\delta_{y_j}\quad\Longrightarrow\quad
 \sum_{j,\ell}\gamma_j\gamma_\ell L(y_j,y_\ell)
            =\int_{\mathbb R}\operatorname{Var}_{\nu_k}(g)\,dk.
\]
Continuity of $L$ justifies the endpoint limit, while the spectral mass
of a fixed interior compact interval has exponential tails in $k$ and
dominates the smoothed variances. For
$g=\sum_i v_i\mathbf1_{I_i}$, the argument principle gives
$\sum_i v_iN_{i,h}=\sum_j\gamma_jW_h(y_j)$.
Expanding $g^2$ in the spectral variance yields the intersection term
in \eqref{thin:eq:band-sigma}, so the identity covers overlapping and
coincident bands. It also identifies the nullspace: every $\nu_k$ has
a strictly positive density, and $g$ vanishes outside an interior
compact set, so its variance is zero exactly when $g=0$ almost
everywhere. For disjoint nonempty bands the intersection term vanishes
and the product of their positive spectral masses makes the covariance
strictly negative.

The Gaussian limit is proved using the annular field itself. Normalize
$Z_{h,y}(x)=Y_h(y+ix)/\sqrt{A_h(2y)}$ and expand its centered log
modulus in orthogonal Laguerre polynomials:
\[
 \log|Z|-\E\log|Z|
       =-\frac12\sum_{m\ge1}\frac{\mathsf L_m(|Z|^2)}m
                      \quad\hbox{in }L^2,
\]
where $\mathsf L_m$ is the degree-$m$ Laguerre polynomial orthonormal
for the exponential distribution, with $\mathsf L_m(0)=1$. Circular averaging is
$L^2$-differentiable by the Jensen bounds, and bounded orthogonal
projections commute with that derivative. Hence the winding has the
exact orthogonal expansion
\[
 \begin{gathered}
 \sqrt h\{W_h(y)-\E W_h(y)\}=\sum_{m\ge1}X_{m,h}(y)
                         \quad\hbox{in }L^2,\\
 X_{m,h}(y)=\frac1{\sqrt{2\pi T}}\int_0^T
             \frac1{2m}\partial_y\mathsf L_m(|Z_{h,y}(x)|^2)\,dx.
 \end{gathered}
\]
The deterministic winding shift from the analytic multiplier disappears
after centering. Differentiating the two-variable covariance of each
projection before setting the levels equal retains the cancellation
$S_{h,y},S_{h,z}=O(t^2)$. Since $S_h\le e^{-ct^2}$ near coincidence,
the integrand in the projection covariance satisfies, for $m\ge3$
at equal levels,
\[
 \frac{|\partial_y\partial_z S_h^m|}{m^2}
 \le C\{m^{-1}e^{-c(m-1)t^2}+t^4e^{-c(m-2)t^2}\}.
\]
Integrating this bound gives $O(m^{-3/2})$. Away from coincidence,
the uniform correlation gap and exponential longitudinal decay give
an exponentially decreasing bound in $m$; the first two degrees are
bounded separately. Thus $\Var X_{m,h}(y)\le Cm^{-3/2}$.
Orthogonality across chaos degrees and Cauchy--Schwarz across the
finitely many endpoint levels then give
\[
 \E\left|\sum_j\gamma_j\sum_{m>M}X_{m,h}(y_j)\right|^2
       \le C\left(\sum_j|\gamma_j|\right)^2M^{-1/2},
\]
uniformly for small $h$. This summable bound allows a finite chaos
approximation throughout the thinning limit.

For fixed $M$, each projected winding combination is a normalized
integral of a finite Gaussian polynomial in field values and radial
derivatives at the endpoint levels. The Gaussian vector of these values
and derivatives at the distinct levels has a nondegenerate limiting
covariance matrix because the corresponding spectral functions
$e^{-y_jk}$ and $ke^{-y_jk}$ are linearly independent at distinct
levels. Whitening therefore uses uniformly bounded matrices and
preserves the exponential longitudinal correlation bound.
In a cumulant of order $p$, Wick's formula leaves only connected
pairing diagrams. Each such diagram contains a spanning tree;
successively integrating its leaves costs a bounded factor per edge,
and only the root integration contributes $T$. The normalization
$T^{-p/2}$ therefore gives
\[
 |\kappa_p(U_{M,h})|\le C_{p,M,\gamma}T^{1-p/2},
\]
where $U_{M,h}=\sum_j\gamma_j\sum_{m\le M}X_{m,h}(y_j)$ and
$\kappa_p$ denotes the cumulant of order $p\ge2$.
At fixed $M$, cumulants of orders at least three vanish and the second
converges. The moment--cumulant formula gives convergence to the
Gaussian moments and hence a Gaussian limit. We first let $h\downarrow0$ at
fixed $M$, and then let $M\to\infty$ using the uniform $L^2$ tail.
Applied to every fixed coefficient vector $v$, this proves the joint
limit with covariance $\Sigma$, including its degenerate directions.
Finally, differentiating the diagonal periodization formula gives the
mean winding. Its deterministic constant cancels between band
endpoints, and its exponentially small remainder is negligible after
multiplication by $\sqrt h$, giving the centering in
\eqref{thin:eq:band-means}.

Section~\ref{sec:important} gives the analytic setup and boundary
normalizations. Section~\ref{sec:fixedq-covariance} derives the winding
covariance formulas, and Section~\ref{sec:fixedq-comparison} proves the
fixed-annulus comparison and Theorem~\ref{thm:joint}.
Sections~\ref{sec:fixedq-corrections} and~\ref{sec:fixedq-rates} give the finite
covariance corrections, the proof of Theorem~\ref{thm:sharp}, and the
normal approximation bounds. Section~\ref{thin:sec:counts} develops the
periodized strip analysis and proves Theorem~\ref{thin:thm:bands}.

\section{The Gaussian Laurent series and winding numbers}
\label{sec:important}\label{sec:fixedq-main}

Fix $q\in(0,1)$. The Gaussian variables introduced in
Section~\ref{sec:introduction} have density $\pi^{-1}e^{-|z|^2}$, and
the Laurent series is
\[
 c_n=(1+q^{2n+1})^{-1/2},\qquad
 F(z)=\sum_{n\in\mathbb Z}c_n\xi_nz^n,\quad q<|z|<1.
\]
In the fixed-annulus results, implicit constants may depend on $q$.
On each compact annulus $q<R_0\le |z|\le R_1<1$, the absolute sum
of the nonnegative modes is bounded by
$\sum_{n\ge0}|\xi_n|R_1^n$ and the negative modes by
\[
 q^{-1/2}\sum_{j\ge1}|\xi_{-j}|(q/R_0)^j.
\]
Both sums have finite expectation and are therefore finite almost surely.
The same estimates, with polynomial factors in $|n|$, give normal
convergence of every derivative. A countable compact exhaustion therefore
shows that the series defines an analytic function throughout the annulus
almost surely. Its Laurent coefficient $c_0\xi_0$ is nonzero almost surely,
so the function is not identically zero.

The series also converges in $L^2$ at each point, so its covariance
kernel is obtained by taking expectations term by term.
Rotational invariance follows because multiplication by deterministic
phases preserves the joint law of the independent Gaussian coefficients.
The convergence estimates also apply to the disk comparison series
$G_+$ and $G_-$.

We use $\Li(u)=\sum_{j\ge1}u^j/j^2$ for $0\le u\le1$.
The logarithmic distances to the two boundary components are
\[
 \ell_-=\log(s/q),\qquad \ell_+=-\log t.
\]
Since $v(e^{-\ell})\sim(4\ell)^{-1}$ as $\ell\downarrow0$,
we have $v(q/s)\sim(4\ell_-)^{-1}$ and $v(t)\sim(4\ell_+)^{-1}$.
The convergence in Theorem~\ref{thm:joint} can therefore also be written as
\[
 \bigl(2\sqrt{\ell_-}(N_--\E N_-),\;
       2\sqrt{\ell_+}(N_+-\E N_+)\bigr)
 \Longrightarrow\mathcal N(0,I_2).
\]
In terms of radial distance, the variance equivalents are
$q/[4(s-q)]$ for the inner count and $1/[4(1-t)]$ for the outer count.
The two logarithmic distances may tend to zero at arbitrary relative rates.

\section{Covariances of winding numbers}\label{sec:fixedq-covariance}

\subsection{Logarithmic covariance and radial differentiation}

For centered jointly proper complex Gaussian analytic functions $f,g$,
write, at points with positive diagonal variances,
\[
 K_{fg}(z,w)=\E f(z)\overline{g(w)},\qquad
 S_{fg}(z,w)=\frac{|K_{fg}(z,w)|^2}{K_{ff}(z,z)K_{gg}(w,w)}.
\]

The log-modulus covariance identity below is given in
\cite[Lemma~3.5.2]{HKPV}.

\begin{lemma}\label{lem:logcov}
Let $X,Y$ be centered jointly proper complex Gaussian variables with
positive variances. Then
\[
 \Cov(\log|X|,\log|Y|)=\tfrac14\Li\!\left(
 \frac{|\E X\overline Y|^2}{\E|X|^2\E|Y|^2}\right).
\]
\end{lemma}

\begin{proof}
Normalize the variances to one, so that $|X|^2$ and $|Y|^2$ are exponential
with mean one.
Write
\[
 G_t(u)=\frac1{1-t}\exp\left(-\frac{ut}{1-t}\right)
       =\sum_{m\ge0}L_m(u)t^m,\qquad |t|<1.
\]
For $s,t$ near zero, integration gives
\[
 \int_0^\infty G_t(u)G_s(u)e^{-u}\,du=\frac1{1-st},
\]
so the Laguerre polynomials are orthonormal in $L^2(e^{-u}du)$.
To prove completeness, let $\psi\in L^2(e^{-u}du)$ be orthogonal to
every polynomial. Its transform $\int_0^\infty \psi(u)e^{-u}e^{tu}\,du$ is
analytic for $\Re t<1/2$ and all its derivatives at zero vanish.
The identity theorem and uniqueness of the Fourier transform of an
$L^1$ function then give $\psi=0$.
Integrating the generating function against the logarithm gives
\[
 \int_0^\infty G_t(u)e^{-u}\log u\,du
       =-\gamma+\log(1-t),
\]
where $\gamma$ is Euler's constant. Thus the Laguerre coefficients of
$\log u$ are $-\gamma$ for $m=0$ and $-1/m$ for $m\ge1$, and
\[
 \log|X|=-\gamma/2-\tfrac12\sum_{m\ge1}m^{-1}L_m(|X|^2)
\]
in $L^2$.

If $\rho=\E X\overline Y$, the Gaussian Laplace integral gives
\[
 \E G_t(|X|^2)G_s(|Y|^2)=\frac1{1-st|\rho|^2}.
\]
Comparing coefficients yields
\[
 \E L_m(|X|^2)L_n(|Y|^2)=\mathbf1_{m=n}|\rho|^{2m}.
\]
Taking the covariance of the two $L^2$ expansions gives the identity.
When $|\rho|=1$, the normalized variables satisfy $|X|=|Y|$ almost surely;
the identity follows from the same expansion and $\sum m^{-2}<\infty$.
\end{proof}

Consequently, for nondegenerate centered jointly proper Gaussian fields,
\begin{equation}\label{eq:logcov}
\Cov(\log|f(z)|,\log|g(w)|)=\tfrac14\Li(S_{fg}(z,w)).
\end{equation}

\begin{lemma}\label{lem:covariance}
For $(f,g)=(F,F),(F,G_+)$ or $(G_+,G_+)$, write the covariance kernels as
\[
 K_{fg}(z,w)=C(z\bar w),\quad
 K_{ff}(z,z)=A(|z|^2),\quad
 K_{gg}(w,w)=B(|w|^2).
\]
Set
\[
 S(x,y,\theta)=\frac{|C(\sqrt{xy}e^{i\theta})|^2}{A(x)B(y)}.
\]
Then, for radii strictly inside their analytic domains,
\begin{equation}\label{eq:windingcov}
 \Cov(W_f(\sqrt x),W_g(\sqrt y))
 =\frac{xy}{2\pi}\int_{-\pi}^{\pi}
        \partial_x\partial_y\Li(S(x,y,\theta))\,d\theta.
\end{equation}
Where $0<S<1$, put $L=\log S$. The differentiated dilogarithm in
\eqref{eq:windingcov} is
\begin{equation}\label{eq:T}
 T=\frac{S}{1-S}L_xL_y-\log(1-S)L_{xy}.
\end{equation}
At zeros of $C$, this derivative is given by
\[
 T=\Li''(S)S_xS_y+\Li'(S)S_{xy}
\]
with the continuous values of $\Li'$ and $\Li''$ at $S=0$.

For every $q<s<c<t<1$, the covariance matrix of the two counts
$Z_F\{s<|z|<c\}$ and $Z_F\{c<|z|<t\}$ is positive definite.
\end{lemma}

\begin{proof}
Let $J_f(r)=(2\pi)^{-1}\int\log|f(re^{i\theta})|\,d\theta$.
Applying \eqref{eq:logcov} and rotational invariance yields
\[
 \Cov(J_f(r),J_g(s))
 =\frac1{8\pi}\int_{-\pi}^{\pi}\Li(S(r^2,s^2,\theta))\,d\theta.
\]
Put $j_f(u)=J_f(e^u)$. Minkowski's inequality gives
$j_f(u)\in L^2$, since the log modulus of a nondegenerate complex
Gaussian has finite second moment. We next justify the identity
$j_f'(u)=W_f(e^u)$ in $L^2$.

First, every deterministic circle in the domain is almost surely free of
zeros. Indeed, the diagonal variance is bounded above and away from zero
on compact subsets, so the log moduli have uniformly bounded first
absolute moments there. We may take expectations in
$\Delta\log|f|=2\pi\mathcal Z_f$, where $\mathcal Z_f$ is the
zero-counting measure. This gives the smooth expected zero density
$(4\pi)^{-1}\Delta\log K_{ff}(z,z)$ with respect to area measure.
A circle has area zero and hence contains no zeros almost surely.

Annular Jensen's formula makes $j_f$ convex. At a radius whose circle
contains no zeros, its derivative is the winding number; crossing a
radius containing zeros increases the slope by their total multiplicity.
For a fixed logarithmic radius $u$ and a fixed $h>0$ such that both
neighboring circles lie in the domain, convexity gives almost surely
\[
 \frac{j_f(u)-j_f(u-h)}h\le W_f(e^u)
 \le\frac{j_f(u+h)-j_f(u)}h.
\]
Both secant quotients belong to $L^2$, so $W_f(e^u)\in L^2$.

Choose fixed logarithmic radii $u_0<u_1<u<u_2<u_3$ in the analytic domain.
For $u'\in(u_1,u_2)$, monotonicity of the derivative of $j_f$
bounds $W_f(e^{u'})$ between
\[
 \frac{j_f(u_1)-j_f(u_0)}{u_1-u_0}
 \quad\hbox{and}\quad
 \frac{j_f(u_3)-j_f(u_2)}{u_3-u_2}.
\]
Both bounds are in $L^2$. Almost surely, a compact band about the fixed
zero-free circle contains only finitely many zeros, so
$W_f(e^{u'})\to W_f(e^u)$ as $u'\to u$.
The secant bounds dominate these windings in absolute value by a single
$L^2$ random variable. Dominated convergence gives $L^2$ continuity,
and the integrated Jensen identity gives $j_f'(u)=W_f(e^u)$ in $L^2$.
Equivalently, $J_f'(r)=W_f(r)/r$ in $L^2$.

We now differentiate the angular covariance integral. For a
self-covariance, let $d_n>0$ be the kernel coefficients. Keeping only
the $n=0,1$ modes in the Gram determinant gives
\[
 A(x)A(y)-|C(\sqrt{xy}e^{i\theta})|^2
 \ge d_0d_1\bigl(x+y-2\sqrt{xy}\cos\theta\bigr).
\]
On a compact neighborhood of a fixed interior squared radius this implies
\[
 c\bigl((x-y)^2+\theta^2\bigr)\le1-S
 \le C\bigl((x-y)^2+\theta^2\bigr)
\]
near the diagonal. The upper bound and
$S_x,S_y=O(\sqrt{(x-y)^2+\theta^2})$, $S_{xy}=O(1)$
follow from smoothness, since $S$ attains its maximum one on the diagonal.
We also have
\[
 |\Li'(S)|\le C(1+|\log(1-S)|),\qquad
 |\Li''(S)|\le C(1-S)^{-1}
\]
there. The formula for $T$ in terms of $S_x,S_y,S_{xy}$ therefore gives
\[
 |\partial_x\partial_y\Li(S)|
 \le C(1+|\log|\theta||),\qquad\theta\ne0,
\]
uniformly over that radial neighborhood.
Away from the diagonal, the formula in terms of $S_x,S_y,S_{xy}$
is smooth, including at zeros of $C$.
For $(F,G_+)$, the negative Laurent component is independent of $G_+$
and gives a uniform correlation gap on every compact radial
neighborhood. The resulting bounds justify differentiation under the
integral for all three pairs.
Differentiating the covariance of $J_f$ and $J_g$ in both radii and
using $r\partial_r=2x\partial_x$ gives \eqref{eq:windingcov}.
Finally, the differentiation rule
$d\Li(S)/d\log S=-\log(1-S)$ gives \eqref{eq:T}.

For the positive-definiteness assertion, denote the counts with split
radius $c$ by $N_-$ and $N_+$. Both have finite second moments by the
winding bounds above. Put $A=\{s\le|z|\le t\}$. The tail of $F$
outside the indices $-M,\ldots,M$ has expected uniform norm at most
\[
 \frac{\sqrt\pi}{2}\left\{\frac{t^{M+1}}{1-t}
 +\frac{\sqrt q}{s}\frac{(q/s)^M}{1-q/s}\right\}\longrightarrow0,
\]
since $c_n\le1$ for $n\ge0$ and $c_{-m-1}\le q^{m+1/2}$ for $m\ge0$.
Given a polynomial $f$ of degree at most one and $\varepsilon>0$,
choose $M$ large enough that the tail norm is less than $\varepsilon/2$
with probability at least $1/2$, using Markov's inequality.
The coefficients in the finite block have an everywhere positive joint
density, so their Laurent sum approximates $f$ uniformly on $A$ within
$\varepsilon/2$ with positive probability. Independence of the block
and the tail then gives
$\mathbb P(\sup_A|F-f|<\varepsilon)>0$.

Apply this approximation to $1$, $z-a_-$ and $z-a_+$ with real
$s<a_-<c<a_+<t$.
For each polynomial, choose $\varepsilon$ below half its minimum
modulus on the circles of radii $s,c,t$. The linear homotopy from that
polynomial to $F$ is nonzero on all three circles and therefore
preserves their winding numbers.
The annular argument principle then gives positive probabilities for
$(N_-,N_+)=(0,0),(1,0),(0,1)$.
A real linear combination of the counts with zero variance would be
constant almost surely. The three events force the constant and both
coefficients to vanish.
Thus the covariance matrix is positive definite.
\end{proof}

\section{Comparison with independent disk fields}\label{sec:fixedq-comparison}

\subsection{The disk counting law}

For a disk GAF with independent standard complex Gaussian coefficients,
Peres--Vir\'ag~\cite[Theorem~2(i)]{PV} proved that, for every $0<r<1$,
\begin{equation}\label{eq:disklaw}
 M_\pm(r)\ \stackrel{d}=\ \sum_{k\ge1}B_{k,r},\qquad
 \Pr(B_{k,r}=1)=r^{2k},
\end{equation}
where the Bernoulli variables in each sum are independent.
Summing the Bernoulli means and variances gives
$\E M_\pm(r)=m(r)$ and $\Var M_\pm(r)=v(r)$, where
$m(r)=r^2/(1-r^2)$ and $v(r)=r^2/(1-r^4)\to\infty$ as $r\uparrow1$.

\subsection{Kernel decomposition and correlation gaps}

Write $K_G$ for the common covariance kernel of $G_+$ and $G_-$.
Near the circle $|u|=1$,
\begin{align*}
 K_F(u)&=\sum_{n\in\mathbb Z}c_n^2u^n=\frac1{1-u}+B_F(u),\\
 K_G(u)&=\frac1{1-u},\\
 K_c(u)&=\E F(z)\overline{G_+(w)}
       =\sum_{n\ge0}c_nu^n
       =\frac1{1-u}+B_c(u),\quad u=z\bar w.
\end{align*}
Here $B_F$ is holomorphic on $q^2<|u|<q^{-2}$,
and $B_c$ is holomorphic on $|u|<q^{-2}$.
Their Laurent and Taylor coefficients, respectively, are real.
On a fixed neighborhood of the unit circle, both functions are bounded
together with any fixed number of derivatives.
These properties follow from $c_n^2-1=O_q(q^{2n})$ and
$c_n-1=O_q(q^{2n})$ for $n\ge0$, and
$c_{-k}^2=O_q(q^{2k})$ for $k\ge1$.

For the boundary estimates, put $a=r^2\uparrow1$, $\delta=1-a$, and
\[
 S_0(a,\theta)=\frac{\delta^2}{|1-ae^{i\theta}|^2}
 =\frac{\delta^2}{\delta^2+2a(1-\cos\theta)}.
\]
Let $S_F(a,\theta)$ denote the squared normalized self-correlation
of $F$ on the circle of radius $r$, and let $S_c(a,\theta)$ denote
the squared normalized cross-correlation of $F$ and $G_+$.

\begin{lemma}\label{lem:gaps}
For $a$ sufficiently close to one, the following bounds hold uniformly
for $\theta\in[-\pi,\pi]$, with positive constants depending only on $q$:
\begin{align}
 1-S_F(a,\theta)&\ge c_q\bigl(1-S_0(a,\theta)\bigr),
 \label{eq:selfgap}\\
 1-S_c(a,\theta)&\ge c_q\bigl(\delta+1-S_0(a,\theta)\bigr).
 \label{eq:crossgap}
\end{align}
Moreover, for a fixed small $\theta_0>0$,
\[
 S_F(a,\theta),S_c(a,\theta)
 \le C_q\frac{\delta^2}{\delta^2+\theta^2},
 \qquad |\theta|\le\theta_0,
\]
and both correlations are $O_q(\delta^2)$ for
$\theta_0\le|\theta|\le\pi$.
\end{lemma}

\begin{proof}
For a probability mass function $(p_n)$, write
$\phi(\theta)=\sum p_ne^{in\theta}$. Then
\[
 1-|\phi(\theta)|^2
 =\sum_{n,m}p_np_m\bigl(1-\cos((n-m)\theta)\bigr).
\]
The geometric masses $g_n=\delta a^n$, $n\ge0$,
have characteristic function $\delta/(1-ae^{i\theta})$, whose
squared modulus is $S_0$.
For $F$, the masses
\[
 p_n=\frac{c_n^2a^n}{K_F(a)},\quad n\in\mathbb Z,
\]
satisfy $p_n\ge c_qg_n$ for $n\ge0$,
since $c_n^2\ge(1+q)^{-1}$ and
$K_F(a)=\delta^{-1}+O_q(1)$.
Restricting the nonnegative sum over $(n,m)$ to $n,m\ge0$ and using
this domination gives \eqref{eq:selfgap}.

For the cross-correlation, put
\[
 V_+(a)=\sum_{n\ge0}c_n^2a^n,\quad
 V_-(a)=\sum_{n<0}c_n^2a^n,\quad
 \kappa=\frac{K_c(a)^2}{K_F(a)K_G(a)}.
\]
For $n\ge0$, the masses $p_n=c_na^n/K_c(a)$ also dominate $c_qg_n$.
Their characteristic function $\phi$ therefore satisfies
$S_c(a,\theta)=\kappa|\phi(\theta)|^2$ and
$1-|\phi|^2\ge c_q(1-S_0)$.
Cauchy--Schwarz gives $K_c(a)^2\le V_+(a)K_G(a)$, hence
\[
 1-\kappa\ge\frac{V_-(a)}{K_F(a)}\ge c_q\delta.
\]
The kernel decompositions also give $\kappa=1+O_q(\delta)$, so
$\kappa\ge1/2$ near the boundary. Combining these estimates in
$1-S_c=(1-\kappa)+\kappa(1-|\phi|^2)$ proves \eqref{eq:crossgap}.
The upper bounds follow from the kernel decompositions and
$K_F(a),K_G(a)=\delta^{-1}+O_q(1)$.
\end{proof}

\subsection{Asymptotics of the covariance integrals}

We apply Lemma~\ref{lem:covariance} to the pairs
$(F,F)$, $(F,G_+)$, and $(G_+,G_+)$.
In each case, the kernel $C(u)$ has the form $(1-u)^{-1}+H(u)$,
where $H$ is analytic near the unit circle. The diagonal variance
functions $A(x)$ and $B(y)$ have analogous decompositions.

\begin{lemma}\label{lem:asymptotic}
For $(f,g)=(F,F)$ or $(F,G_+)$, as $r\uparrow1$,
\[
 \Cov(W_f(r),W_g(r))\sim\frac1{2(1-r^2)}.
\]
\end{lemma}

\begin{proof}
Fix $\theta_0$ small enough that the kernels $C(u)$ do not vanish
when $u$ is near $1$ and $|\arg u|\le\theta_0$. Write
\[
 C(u)=\frac1{1-u}\exp D_C(u),\qquad
 A(x)=\frac1{1-x}\exp D_A(x),
\]
\[
 B(y)=\frac1{1-y}\exp D_B(y),
\]
where the analytic corrections are real on the real axis and uniformly
bounded together with their first few derivatives.
For $u=\sqrt{xy}e^{i\theta}$, the logarithm $L=\log S$ satisfies
\begin{align*}
 L_x&=\frac1x\Re\frac{u}{1-u}-\frac1{1-x}
       +\frac1x\Re(uD_C'(u))-D_A'(x),\\
 L_y&=\frac1y\Re\frac{u}{1-u}-\frac1{1-y}
       +\frac1y\Re(uD_C'(u))-D_B'(y),\\
 L_{xy}&=\frac1{2xy}\Re\frac{u}{(1-u)^2}+O_q(1).
\end{align*}
At $x=y=a$, the disk first derivatives are
\begin{equation*}
 L_{0,x}=L_{0,y}
 =-\frac{(1+a)(1-\cos\theta)}{\delta|1-ae^{i\theta}|^2}.
\end{equation*}
For a self-covariance, the correction to each first derivative vanishes
at $\theta=0$. The real coefficients make the correction even in
$\theta$, so it is $O_q(\theta^2)$. For the cross-covariance, the
correction is $O_q(1)$.
Thus, for $|\theta|\le\theta_0$,
\begin{align*}
 |L_x|+|L_y|
 &\le C_q\left(\frac{\theta^2}{\delta(\delta^2+\theta^2)}+1\right),\\
 |L_{xy}|&\le C_q(\delta^2+\theta^2)^{-1}+C_q,
\end{align*}
and in the self-covariance case the constant term in the first bound
can be replaced by $\theta^2$.

Set $\theta=\delta u$ with $u$ real. For every fixed $u\ne0$,
the kernel decompositions give
\begin{align*}
 S(a,a,\delta u)&\longrightarrow (1+u^2)^{-1},\\
 \delta L_x,\delta L_y&\longrightarrow-\frac{u^2}{1+u^2},\\
 \delta^2L_{xy}&\longrightarrow\frac{1-u^2}{2(1+u^2)^2}.
\end{align*}
Hence the rescaled integrand $\delta^2T(a,a,\delta u)$ tends to
\begin{equation}\label{eq:limitT}
 \mathcal T(u)=\frac{u^2}{(1+u^2)^2}
 -\log\left(\frac{u^2}{1+u^2}\right)
       \frac{1-u^2}{2(1+u^2)^2}.
\end{equation}

To apply dominated convergence, we estimate the rescaled integrand
separately for $|\theta|\le\delta$, $\delta\le|\theta|\le\theta_0$,
and $|\theta|\ge\theta_0$.
For the cross-covariance on $|\theta|\le\delta$,
Lemma~\ref{lem:gaps} gives
\[
 \frac{S}{1-S}\le C_q\frac{\delta^2}{\theta^2+\delta^3}.
\]
Together with the first-derivative estimates, this bounds the rational
term in \eqref{eq:T} by
\[
 C_q\left(\frac{\theta^2}{\delta^4}
       +\frac{\delta^2}{\theta^2+\delta^3}\right).
\]
After multiplication by $\delta^2$ and rescaling,
this is at most $C_q(u^2+\delta)$.
For the self-covariance, the stronger first-derivative cancellation and
\eqref{eq:selfgap} give the bound $C_qu^2$.
For both self- and cross-covariances, the logarithmic term, multiplied
by $\delta^2$, is bounded by $C_q(1+|\log|u||)$ for $0<|u|\le1$.

For $\delta\le|\theta|\le\theta_0$, Lemma~\ref{lem:gaps} gives
$1-S\ge c_q$ and $S\le C_q\delta^2/\theta^2$.
It follows that
\[
 \delta^2|T(a,a,\delta u)|\le C_q(u^{-2}+u^{-4}),
 \qquad |u|\ge1.
\]
For $|\theta|\ge\theta_0$, write
$T=\Li''(S)S_xS_y+\Li'(S)S_{xy}$, which remains valid at zeros
of the covariance kernel.
The kernels and their derivatives are bounded,
$S=O_q(\delta^2)$, $S_x,S_y=O_q(\delta)$ and $S_{xy}=O_q(1)$,
so $T=O_q(1)$.
Since $|u|\le\pi/\delta$, this gives
$\delta^2|T(a,a,\delta u)|\le C_qu^{-2}$ on the outer angular region.

Extend the rescaled integrand by zero outside $|u|\le\pi/\delta$.
The preceding bounds give an integrable majorant independent of $a$.
Dominated convergence in \eqref{eq:windingcov} gives
\[
 \delta\,\Cov(W_f(r),W_g(r))
 \longrightarrow\frac1{2\pi}\int_{\mathbb R}\mathcal T(u)\,du.
\]
To evaluate the limiting integral, let $h(u)=u/(1+u^2)$.
The first summand in \eqref{eq:limitT} integrates to $\pi/2$.
Integration by parts in the second summand gives
\[
 -\tfrac12\int_{\mathbb R}
       \log\frac{u^2}{1+u^2}\,h'(u)\,du
 =\tfrac12\int_{\mathbb R}
       h(u)\frac{2}{u(1+u^2)}\,du=\frac\pi2.
\]
The boundary terms vanish at zero and infinity.
Thus the integral equals $\pi$, proving the asymptotic.
\end{proof}

\begin{proposition}\label{thm:coupling}
As $r\uparrow1$,
\[
 \E\left|\bigl(W_F(r)-\E W_F(r)\bigr)
       -\bigl(M_+(r)-m(r)\bigr)\right|^2
 =o\bigl((1-r^2)^{-1}\bigr).
\]
Moreover $\E W_F(r)-m(r)=O_q(1)$ and $\Var W_F(r)\sim v(r)$.
The transformed Laurent GAF
\[
 \widetilde F(w)=\frac{\sqrt q}{w}F(q/w)
\]
satisfies the same assertions with $M_-(r)$ in place of $M_+(r)$.
\end{proposition}

\begin{proof}
Lemma~\ref{lem:asymptotic} and the exact disk variance imply
\begin{align*}
 \Var(W_F(r)-M_+(r))
 &=\Var W_F(r)+\Var M_+(r)-2\Cov(W_F(r),M_+(r))\\
 &=o(\delta^{-1}).
\end{align*}
Differentiating the expected log modulus gives
\[
 \E W_F(r)=a\frac{K_F'(a)}{K_F(a)}.
\]
The decomposition $K_F(a)=\delta^{-1}+B_F(a)$ gives
$K_F'(a)/K_F(a)=\delta^{-1}+O_q(1)$, proving the mean bound.

For the inner boundary, direct index reversal gives
\[
 \widetilde F(w)=\sum_{k\ge0}c_k\xi_{-k-1}w^k
 +\sum_{n\ge0}
       \frac{q^{n+1/2}\xi_n}{\sqrt{1+q^{2n+1}}}\,w^{-n-1}.
\]
Because $q^{n+1/2}/\sqrt{1+q^{2n+1}}=c_{-n-1}$ and the reversed
coefficient sequence is again independent and standard Gaussian,
$\widetilde F$ has the same law as $F$. Its coupling with $G_-$ uses
$(\xi_{-k-1})_{k\ge0}$ with the weights $c_k$, exactly as in the outer
comparison. Applying that comparison to $\widetilde F$ proves the
assertions for the inner boundary.
\end{proof}

\subsection{Proof of the joint limit}

\begin{proof}[Proof of Theorem~\ref{thm:joint}]
The argument principle and inversion give
\[
 N_+(t)=W_F(t)-W_F(\sqrt q),\qquad
 W_{\widetilde F}(r)=-1-W_F(q/r),
\]
and consequently
\[
 N_-(s)=W_{\widetilde F}(q/s)-W_{\widetilde F}(\sqrt q).
\]
Lemma~\ref{lem:covariance} gives finite second moments for both
windings on the fixed middle circle.
Proposition~\ref{thm:coupling}, the bounded mean differences,
and $v(r)\asymp(1-r^2)^{-1}$ imply
\[
 \frac{N_+(t)-M_+(t)}{\sqrt{v(t)}}
 \longrightarrow0\quad\hbox{in }L^2,
\]
\[
 \frac{N_-(s)-M_-(q/s)}{\sqrt{v(q/s)}}
 \longrightarrow0\quad\hbox{in }L^2.
\]
For each disk count, use the Bernoulli representation
\eqref{eq:disklaw} and put $X_{k,r}=B_{k,r}-r^{2k}$. Then
$|X_{k,r}|\le1$ and
$\sum_k\E|X_{k,r}|^3\le v(r)$.
For each fixed real $u$, Taylor expansion of the characteristic-function
factors and their logarithms gives a summable remainder and yields
\[
 \log\E\exp\left(iu\frac{\sum_kX_{k,r}}{\sqrt{v(r)}}\right)
 =-u^2/2+O_u(v(r)^{-1/2}).
\]
For $r$ sufficiently close to one, each characteristic-function factor is
$1+O_u(\operatorname{Var}(B_{k,r})/v(r))$, uniformly as $r\uparrow1$,
so we use the branch of its logarithm near one.
The squared deviations of these factors from one sum to
$O_u(v(r)^{-1})$. Apply the expansion first to finite products, then
pass to the infinite product using the summable tails. The resulting
characteristic function converges to that of a standard normal.

Since the random functions $G_+$ and $G_-$ are independent,
their centered, normalized counts converge jointly to independent standard
normals as $s\downarrow q$ and $t\uparrow1$ at arbitrary relative rates.
The $L^2$ approximations transfer this limit to the annular counts
by Slutsky's theorem.

The expectation bounds in Proposition~\ref{thm:coupling}, together
with the fixed middle-circle windings, give the mean asymptotics.
Centering the $L^2$ approximations and applying Cauchy--Schwarz gives
the variance asymptotics. Replacing the disk means and variances by
the exact annular means and variances therefore preserves the limit.
\end{proof}

\section{Exact mean and covariance corrections}\label{sec:fixedq-corrections}

\subsection{Mean values}

\begin{proposition}\label{prop:means}
For the Gaussian Laurent series $F$,
\[
 \E W_F(\sqrt q)=-\tfrac12,\qquad
 \E W_F(r)=m(r)+O_q(1-r^2)\quad(r\uparrow1).
\]
Consequently,
\[
 \E N_+(t)=m(t)+\tfrac12+O_q(1-t^2),\qquad
 \E N_-(s)=m(q/s)+\tfrac12+O_q(1-q^2/s^2).
\]
Thus either $m(\cdot)$ or $m(\cdot)+1/2$ may be used to center
the counts in Theorem~\ref{thm:joint}.
\end{proposition}

\begin{proof}
Put $a=r^2$ and $b_n=c_n^2=(1+q^{2n+1})^{-1}$.
The identity $b_{-n-1}=1-b_n$ gives
\[
 B_F(u)=\sum_{n\ge0}(1-b_n)(u^{-n-1}-u^n),\qquad B_F(1)=0.
\]
The series and its derivatives converge normally on $q^2<|u|<q^{-2}$.
Thus $B_F(a)=O_q(1-a)$ and
\[
 a\frac{K_F'(a)}{K_F(a)}
 =\frac a{1-a}
  +a\partial_a\log[1+(1-a)B_F(a)]
 =\frac a{1-a}+O_q(1-a).
\]
The winding numbers have finite expectations by the Jensen argument
in Lemma~\ref{lem:covariance}. Inversion gives
\[
 W_{\widetilde F}(r)=-1-W_F(q/r).
\]
Since the two fields have the same law, $\E W_F(\sqrt q)=-1/2$.
The argument principle gives the stated expectation for the outer
count; applying it to $\widetilde F$ gives the inner-count expectation.
The differences between the proposed centerings and these expectations
vanish after division by the corresponding diverging standard deviations.
\end{proof}

\subsection{Finite variance and covariance corrections}

Write $K(u)=K_F(u)$ and $c=\sqrt q$ throughout the covariance calculations below.
For $q^2<a<1$ and $q<r<1$, define
\[
 \mathcal B(a)=\frac1{2\pi}\int_{-\pi}^{\pi}
                    |K(ae^{i\theta})|^2d\theta
 =\sum_{n\in\mathbb Z}b_n^2a^{2n},\qquad
 \mathcal A(r)=\frac{\mathcal B(r)}{K(r^2)},
\]
where $b_n=(1+q^{2n+1})^{-1}$.
Define
\[
 \kappa_q=q\mathcal B(q)-\Var W_F(c)
 -\left.\partial_{\log r}\mathcal A(r)\right|_{r=c}.
\]

\begin{proposition}\label{prop:finite}
For every fixed $q\in(0,1)$,
\[
 \Var W_F(r)=v(r)-q\mathcal B(q)+o(1)\quad(r\uparrow1).
\]
As $s\downarrow q,t\uparrow1$, with arbitrary relative rates,
\[
 \Cov(N_-(s),N_+(t))\longrightarrow\kappa_q,
\]
\[
 \Var N_+(t)=v(t)-\kappa_q+o(1),\qquad
 \Var N_-(s)=v(q/s)-\kappa_q+o(1),
\]
\[
 \Var\bigl(N_-(s)+N_+(t)\bigr)=v(q/s)+v(t)+o(1).
\]
In particular,
\[
 \operatorname{Corr}(N_-(s),N_+(t))
 =O_q(\sqrt{(s-q)(1-t)}).
\]
\end{proposition}

\begin{proof}
First let the two winding circles approach opposite boundary components.
The reciprocal diagonal kernel extends smoothly to both radial endpoints.
The outer kernel decomposition and the inversion identity give
\[
 \left.\partial_{\log r}K(r^2)^{-1}\right|_{r=1}=-2,
 \qquad
 \left.\partial_{\log r}K(r^2)^{-1}\right|_{r=q}=2q.
\]
For example,
$K(r^2)^{-1}=(r^2-q^2)/q+O_q((r^2-q^2)^2)$
at the inner endpoint.
As $s\downarrow q$ and $t\uparrow1$, the products $st$ stay in a
compact subset of $(q^2,1)$. Hence $K(st e^{i\theta})$ and its
required radial derivatives are uniformly bounded.
Write
\[
 R(s,t,\theta)
 =K(s^2)^{-1}K(t^2)^{-1}|K(st e^{i\theta})|^2.
\]
The squared normalized correlation $R$ tends to zero uniformly in
$\theta$. Its first logarithmic radial derivatives also tend to zero,
while its mixed logarithmic radial derivative tends to
$-4q|K(qe^{i\theta})|^2$. Since $\Li'(0)=1$, substitution into
Lemma~\ref{lem:covariance} gives
\[
 \Cov(W_F(s),W_F(t))\longrightarrow-q\mathcal B(q).
\]
If one radius $r$ remains in a fixed compact neighborhood of $c$, the
same computation as $t\uparrow1$ gives
\[
 \Cov(W_F(r),W_F(t))\longrightarrow
 -\tfrac12\partial_{\log r}\mathcal A(r).
\]
By inversion, $\Cov(W_F(s),W_F(c))$ has the same limit as
$\Cov(W_F(c),W_F(t))$.
Expand the covariance of $W_F(c)-W_F(s)$ and $W_F(t)-W_F(c)$.
The three limits above give $\kappa_q$, with no restriction on the
relative approach rates.

To obtain the marginal variance corrections, we compare the annular
and disk variance integrands. Set $a=r^2$, $\epsilon=1-a$,
$w=ae^{i\theta}$, $d=|1-w|^2$, and $P(w)=(1-w)^{-1}$.
The cancellation $B_F(1)=0$ in Proposition~\ref{prop:means} gives
$K=P+H$, where
\[
 H(w)=(1-w)h(w)
\]
for an analytic $h$ on a neighborhood of the unit circle.
Define $R=|K(w)|^2/K(a)^2$ and $R_0=\epsilon^2/d$.
For logarithmic radial derivatives, put $D=w\partial_w$ and
\[
 A=2\Re D\log K(w)-2D\log K(a),\qquad
 T=2\Re D^2\log K(w).
\]
The mixed derivative of the dilogarithm with respect to the two
logarithmic radii is
\[
 I_F=\frac{R}{1-R}A^2-\log(1-R)T,
\]
and the winding variance equals $(1/4)\operatorname{av}_\theta I_F$.
For the disk field,
\[
 A_0=-\frac{2a(1+a)(1-\cos\theta)}{\epsilon d},\qquad
 T_0=2\Re\frac{w}{(1-w)^2},\qquad
 I_0=\frac{R_0}{1-R_0}A_0^2-\log(1-R_0)T_0.
\]
Here $\operatorname{av}_\theta=(2\pi)^{-1}\int_{-\pi}^{\pi}$.
Indeed, $\partial_{\log r}\partial_{\log t}
=4r^2t^2\partial_x\partial_y$ in Lemma~\ref{lem:covariance}.

Uniformly as $a\uparrow1$, on a fixed small punctured neighborhood of
$\theta=0$, we have
\begin{equation}\label{eq:finite-gaps}
 |R-R_0|\le C_q\epsilon^2\theta^2/d,\qquad
 \frac{1-R}{1-R_0}=1+O_q(\epsilon^2).
\end{equation}
To prove these bounds, put $Q(w)=1+(1-w)^2h(w)$ and
\[
 E(w)=|K(w)|^2-|P(w)|^2
 =2\Re\left(\frac{1-\overline w}{1-w}\overline{h(w)}\right)
  +d|h(w)|^2.
\]
This expression is uniformly bounded,
and $Q(a)^2-1=\epsilon^2E(a)$.
Substituting these identities gives
\[
 R-R_0=\frac{\epsilon^2}{Q(a)^2}
 \left[E(w)-E(a)+E(a)(1-\epsilon^2/d)\right].
\]
The real part of $(1-\overline w)/(1-w)-1$ is $O(\theta^2/d)$
and its imaginary part is $O(|\theta|(\epsilon+\theta^2)/d)$.
Since $h$ is real on the real axis, the real and imaginary parts of
$h(ae^{i\theta})-h(a)$ are $O_q(\theta^2)$ and $O_q(|\theta|)$,
respectively.
These estimates give $|E(w)-E(a)|\le C_q\theta^2/d$
and prove the first bound in \eqref{eq:finite-gaps}.
Dividing by $1-R_0=2a(1-\cos\theta)/d$ proves the second.

The analytic logarithmic correction also gives
$A-A_0=O_q(\theta^2)$ and $T-T_0=O_q(1)$.
Consequently
\[
 \left|\frac{R}{1-R}A^2-\frac{R_0}{1-R_0}A_0^2\right|\le C_q.
\]
The contribution from the difference of the rational coefficients is bounded by
\[
 C_q\frac{\epsilon^2\theta^2/d}{(\theta^2/d)^2}
 \frac{\theta^4}{\epsilon^2d^2}\le C_q,
\]
and the terms containing $A-A_0$ are bounded by
$C_q\epsilon\theta^2/d+C_q\epsilon^2\theta^2$.
The difference of the logarithmic terms is bounded by
\[
 C_q\epsilon^2/d+C_q\log(1+\epsilon^2/\theta^2).
\]
Thus $|I_F-I_0|$ is bounded by the integrable function
$C_q[1+\log(1+\theta^{-2})]$, independently of $a$.
Outside this angular neighborhood, direct differentiation of $R$
gives a uniform bound, including at zeros of $K$.

For each fixed $\theta\ne0$ with $K(e^{i\theta})\ne0$,
$A,A_0=-2/\epsilon+O_q(1)$, while
$R\sim\epsilon^2|K(e^{i\theta})|^2$ and
$R_0\sim\epsilon^2|P(e^{i\theta})|^2$.
The difference of the logarithmic terms tends to zero, so
\[
 I_F-I_0\longrightarrow
 4\bigl(|K(e^{i\theta})|^2-|P(e^{i\theta})|^2\bigr).
\]
The kernel has only finitely many zeros on the unit circle, since it
is analytic across the circle away from its pole at $1$ and is nonzero
near that pole. At kernel zeros, use the formula in terms of derivatives
of the squared normalized correlation; it satisfies the same integrable
bound.
Dominated convergence therefore gives
\[
 \Var W_F(r)-v(r)\longrightarrow
 \operatorname{av}_\theta
       \bigl(|K(e^{i\theta})|^2-|P(e^{i\theta})|^2\bigr)
 =-2\sum_{n\ge0}b_n(1-b_n)=-q\mathcal B(q).
\]
For the Fourier identity in the last line, first apply Parseval at $a<1$.
The coefficient differences $b_n^2-1$ for $n\ge0$
and the coefficients $b_n^2$ for $n<0$ form absolutely summable sequences.
The expression $E(ae^{i\theta})$ is uniformly bounded,
so its angular integral has the same limit.
Pair $n$ with $-n-1$ and use $b_{-n-1}=1-b_n$
to obtain $-2\sum b_n(1-b_n)$.
The same pairing in $\mathcal B(q)$ gives
$q\mathcal B(q)=2\sum_{n\ge0}b_n(1-b_n)$.

Finally, expanding $\Var(W_F(t)-W_F(c))$ using the fixed-circle
covariance limit gives $v(t)-\kappa_q+o(1)$.
Inversion gives the inner-count formula.
Adding the two variances and twice their covariance gives the total-count formula.
Finite second moments follow from Lemma~\ref{lem:covariance}.
The marginal variances are positive near the boundary by their
asymptotics. Dividing the bounded covariance by their standard
deviations and using
$v(t)\asymp(1-t)^{-1}$ and $v(q/s)\asymp(s-q)^{-1}$ gives the
correlation estimate.
\end{proof}

\subsection{Evaluation and sign of the covariance constant}

\begin{proposition}\label{prop:elliptic}
Let $k\in(0,1)$ be the unique modulus with nome
\[
 q=\exp(-\pi\mathbf K'/\mathbf K),\qquad
 \mathbf K=\int_0^{\pi/2}
       \frac{d\phi}{\sqrt{1-k^2\sin^2\phi}},\qquad
 \mathbf E=\int_0^{\pi/2}\sqrt{1-k^2\sin^2\phi}\,d\phi,
\]
where $k'=\sqrt{1-k^2}$ and
$\mathbf K'=\int_0^{\pi/2}(1-k'^2\sin^2\phi)^{-1/2}d\phi$.
Then
\[
 \Var W_F(\sqrt q)=\frac{\mathbf K\mathbf E}{\pi^2},\qquad
 \left.\partial_{\log r}\mathcal A(r)\right|_{r=\sqrt q}
 =\frac{2\mathbf K\mathbf E}{\pi^2},
\]
\[
 q\mathcal B(q)
 =\frac{\mathbf K(\mathbf E-k'^2\mathbf K)}{\pi^2},\qquad
 \kappa_q
 =-\frac{2\mathbf K\mathbf E+k'^2\mathbf K^2}{\pi^2}<0.
\]
The elliptic nome in these formulas is the inner radius $q$ of the annulus.
\end{proposition}

\begin{proof}
For $v\in[-\mathbf K,\mathbf K]$, define the amplitude
$\phi\in[-\pi/2,\pi/2]$ and the Jacobi functions by
\[
 v=\int_0^\phi\frac{d\alpha}{\sqrt{1-k^2\sin^2\alpha}},\qquad
 \operatorname{sn}(v,k)=\sin\phi,\quad
 \operatorname{cn}(v,k)=\cos\phi,
\]
\[
 \operatorname{dn}(v,k)=\sqrt{1-k^2\sin^2\phi}.
\]
The Fourier expansion~\cite[Eq.~22.11.2]{DLMF} is
\[
 \operatorname{cn}(z,k)=\frac{2\pi}{k\mathbf K}
 \sum_{n\ge0}\frac{q^{n+1/2}
 \cos((2n+1)\pi z/(2\mathbf K))}{1+q^{2n+1}}.
\]
This series converges absolutely for real $z$ and $0<q<1$.
We also use
$\operatorname{cn}'=-\operatorname{sn}\operatorname{dn}$,
$1-\operatorname{cn}^2=\operatorname{sn}^2$, and
$\operatorname{dn}^2=1-k^2\operatorname{sn}^2$,
which follow from the inverse elliptic-integral definition.
Changing variables from $v$ to $\phi$ in this definition gives
$\int_{-\mathbf K}^{\mathbf K}\operatorname{dn}^2(v,k)dv=2\mathbf E$.

Pairing coefficients gives
\[
 K(qe^{i\theta})=2e^{-i\theta/2}
 \sum_{n\ge0}\frac{q^n\cos((n+\tfrac12)\theta)}{1+q^{2n+1}}.
\]
Evaluating the Fourier series at zero and at $z=\mathbf K\theta/\pi$ gives
\[
 \sqrt q K(q)=\frac{k\mathbf K}{\pi},\qquad
 \frac{K(qe^{i\theta})}{K(q)}
 =e^{-i\theta/2}f(\theta),\quad
 f(\theta)=\operatorname{cn}(\mathbf K\theta/\pi,k).
\]
The function $f$ is positive on $(-\pi,\pi)$ and vanishes at $\pm\pi$.
With $D=u\partial_u$ acting on the kernel variable, the phase gives
$\Re D\log K(qe^{i\theta})=-1/2$. Since the diagonal logarithmic
derivative has the same value, both first logarithmic radial derivatives
of the squared normalized self-correlation vanish on the middle circle.
Also $\Re D^2\log K=-\partial_\theta^2\log|K|$, so the covariance
formula reduces to
\[
 \Var W_F(c)=\tfrac12\operatorname{av}_\theta
 \log(1-f^2)\,\partial_\theta^2\log f.
\]
For integration by parts, first remove small neighborhoods of
$0$ and $\pm\pi$.
The boundary product is $O(\theta\log|\theta|)$ at zero
and $O(\pi-|\theta|)$ at either endpoint, so the boundary terms vanish.
It follows that
\[
 \Var W_F(c)
 =\operatorname{av}_\theta\frac{f'(\theta)^2}{1-f(\theta)^2}
 =\left(\frac{\mathbf K}{\pi}\right)^2
   \operatorname{av}_\theta\operatorname{dn}^2(\mathbf K\theta/\pi,k)
 =\frac{\mathbf K\mathbf E}{\pi^2}.
\]
At $0$ and $\pm\pi$, the products are interpreted by their
integrable limiting expressions.
Parseval and the same Fourier identity give
\[
 q\mathcal B(q)
 =\frac{k^2\mathbf K^2}{\pi^2}\operatorname{av}_\theta f^2
 =\frac{\mathbf K(\mathbf E-k'^2\mathbf K)}{\pi^2}.
\]
Here
$\operatorname{av}\operatorname{sn}^2
 =(\mathbf K-\mathbf E)/(k^2\mathbf K)$
follows directly from the preceding integral of $\operatorname{dn}^2$.

For the remaining derivative, write $K_q$ for $K$
to display its dependence on $q$.
The exponentially convergent coefficient sums permit termwise
differentiation at fixed $q$.
The mean winding on the middle circle is $-1/2$, so
\[
 \left.\partial_{\log r}\mathcal A(r)\right|_{r=c}
 =\frac1{K(q)}\sum_{n\in\mathbb Z}(2n+1)b_n^2q^n
 =q\frac d{dq}\log(\sqrt q K_q(q)).
\]
In the total derivative on the right, the coefficients $b_n$ also
depend on $q$. Differentiating gives
$K(q)^{-1}\sum_{n\in\mathbb Z}(n+1/2)b_nq^n(2b_n-1)$.
The terms of $\sum_{n\in\mathbb Z}(n+1/2)b_nq^n$ cancel in the pairs
$n,-n-1$, leaving $K(q)^{-1}\sum_{n\in\mathbb Z}(2n+1)b_n^2q^n$
as in the displayed identity.
By the Fourier identity, this derivative is
$q(d/dq)\log(k\mathbf K)$.

The differentiation and Legendre identities
\cite[Eqs.~19.4.1 and~19.7.1]{DLMF} are
\[
 \frac{d\mathbf K}{dk}
 =\frac{\mathbf E-k'^2\mathbf K}{kk'^2},\qquad
 \mathbf E\mathbf K'+\mathbf E'\mathbf K-\mathbf K\mathbf K'
 =\frac\pi2,
\]
where
$\mathbf E'=\int_0^{\pi/2}(1-k'^2\sin^2\phi)^{1/2}d\phi$.
For $k\in(0,1)$, differentiating
$\log q=-\pi\mathbf K'/\mathbf K$ with these identities gives
\[
 \frac{d\log q}{dk}=\frac{\pi^2}{2kk'^2\mathbf K^2}.
\]
The positive derivative and the endpoint limits of the defining
elliptic integrals show that $k\mapsto q$ is a bijection of $(0,1)$
onto $(0,1)$. Consequently
\[
 q\frac d{dq}\log(k\mathbf K)=\frac{2\mathbf K\mathbf E}{\pi^2}.
\]
Substituting the three evaluated quantities into the definition of
$\kappa_q$ gives the stated formula. Its numerator is positive, so
$\kappa_q<0$.
\end{proof}

\section{Sharp comparison and quantitative normal approximation}
\label{sec:fixedq-rates}

\subsection{The sharp comparison error}

\begin{proof}[Proof of Theorem~\ref{thm:sharp}]
Write $a=r^2=1-\delta$ and set
\[
 D_0=|1-ae^{i\theta}|^2,\quad
 S_0=\frac{\delta^2}{D_0},\quad
 Q_0=\frac{S_0}{1-S_0}=\frac{\delta^2}{2a(1-\cos\theta)}.
\]
Choose a fixed sufficiently small $\theta_0>0$. On
$|\theta|\le\theta_0$, the quantities $D_0$ and $\delta^2+\theta^2$
are uniformly comparable as $\delta\downarrow0$.
Let $S_c$ be the squared normalized cross-correlation of $(F,G_+)$
on the circle of radius $r$, and set
\[
 \kappa=\frac{K_c(a)^2}{K_F(a)K_G(a)}.
\]
The analytic kernel corrections give
\begin{equation}\label{eq:kappa}
 1-\kappa=\beta_q\delta+O_q(\delta^2).
\end{equation}
Indeed, $K_F(a)=\delta^{-1}+B_F(1)+O_q(\delta)$ and
$K_c(a)=\delta^{-1}+B_c(1)+O_q(\delta)$,
so $1-\kappa=(B_F(1)-2B_c(1))\delta+O_q(\delta^2)$.
The coefficient sums give
\[
 B_F(1)-2B_c(1)
 =\sum_{n\ge0}\bigl(c_n^2-1-2(c_n-1)\bigr)
  +\sum_{n<0}c_n^2=\beta_q.
\]

For unequal squared radii, define
\[
 S_0(x,y,\theta)
 =\frac{(1-x)(1-y)}{|1-\sqrt{xy}e^{i\theta}|^2}.
\]
Let $\ell_x$, $\ell_y$, and $\ell_{xy}$ denote the derivatives
$\partial_x\log S_0$, $\partial_y\log S_0$, and
$\partial_x\partial_y\log S_0$, respectively, evaluated at $x=y=a$.
Explicitly,
\[
 \ell_x=\ell_y
 =-\frac{(1+a)(1-\cos\theta)}{\delta D_0},
 \qquad
 \ell_{xy}
 =\frac1{2a^2}\Re\frac{ae^{i\theta}}{(1-ae^{i\theta})^2}.
\]
In particular,
\begin{equation*}
 |\ell_x|\le C\frac{\theta^2}{\delta(\delta^2+\theta^2)},
 \qquad
 |\ell_{xy}|\le C(\delta^2+\theta^2)^{-1}.
\end{equation*}

The analytic corrections have real coefficients, so their real parts
are even in $\theta$. Normalizing at $\theta=0$ therefore gives, on
$|\theta|\le\theta_0$,
\begin{equation}\label{eq:selfstructure}
 S_F=S_0e^{e_F(\theta)},\quad |e_F(\theta)|\le C_q\theta^2,
\end{equation}
\begin{equation}\label{eq:crossstructure}
 S_c=\kappa S_0e^{e_c(\theta)},\quad |e_c(\theta)|\le C_q\theta^2.
\end{equation}
The first radial derivatives of $\log S_F$ differ from those of
$\log S_0$ by $O_q(\theta^2)$; for $\log S_c$, the differences are
$O_q(1)$. The mixed radial derivatives differ from that of $\log S_0$
by $O_q(1)$ in both cases.

\medskip\noindent\textit{The rational term.}
By \eqref{eq:T}, the covariance integrand, with $L=\log S$, is
\[
 T(S)=\frac{S}{1-S}L_xL_y-\log(1-S)L_{xy}.
\]
We first show that the rational terms for $S_F$ and $S_c$ differ
from the disk term by functions with uniformly bounded absolute
integrals.

For the self-correlation, \eqref{eq:selfstructure} implies
\[
 1-S_F=(1-S_0)(1+O_q(\delta^2)),
\]
uniformly on $0<|\theta|\le\theta_0$, since
$S_0\theta^2/(1-S_0)\le C\delta^2$.
Thus
\[
 \left|\frac{S_F}{1-S_F}-Q_0\right|
 \le C_q\frac{S_0\theta^2}{(1-S_0)^2}
 \le C_q\frac{\delta^2(\delta^2+\theta^2)}{\theta^2}.
\]
After multiplication by $\ell_x^2$, this difference is bounded by
$C_q\theta^2/(\delta^2+\theta^2)\le C_q$.
The remaining differences, which contain a correction to a first
radial derivative, are bounded by
\[
 C_qQ_0\bigl(|\ell_x|\theta^2+\theta^4\bigr).
\]
The two terms have integrals $O_q(\delta)$ and $O_q(\delta^2)$,
respectively.

For the cross-correlation, replacing $S_c$ by $\kappa S_0$ in the
rational coefficient gives an error whose product with $\ell_x^2$
is bounded by
\[
 C_q\frac{\delta^2\theta^2(\delta^2+\theta^2)}
              {(\theta^2+\delta^3)^2}
 \frac{\theta^4}{\delta^2(\delta^2+\theta^2)^2}
 =C_q\frac{\theta^6}
 {(\delta^2+\theta^2)(\theta^2+\delta^3)^2}
 \le C_q\frac{\theta^2}{\delta^2+\theta^2}\le C_q.
\]
To replace $\kappa S_0$ by $S_0$, use
\[
 \left|\frac{\kappa S_0}{1-\kappa S_0}-Q_0\right|
 =\frac{(1-\kappa)S_0}{(1-\kappa S_0)(1-S_0)}.
\]
By \eqref{eq:kappa} and the cross-correlation gap,
\[
 \left|\frac{\kappa S_0}{1-\kappa S_0}-Q_0\right|\ell_x^2
 \le C_q\frac{\delta\theta^2}
 {(\delta^2+\theta^2)(\theta^2+\delta^3)}.
\]
The last bound has a uniformly bounded integral on
$|\theta|\le\delta$. Its integral on
$\delta\le|\theta|\le\theta_0$ is at most
$C_q\int_\delta^{\theta_0}\delta\theta^{-2}d\theta$.
For the terms containing corrections to the first radial derivatives
of $\log S_c$, use
\[
 \frac{S_c}{1-S_c}\le C_q\frac{\delta^2}{\theta^2+\delta^3}.
\]
Terms with one factor $\ell_x$ satisfy the preceding bound for the
coefficient difference. The product of two bounded derivative
corrections has integral $O_q(\sqrt\delta)$. Thus the rational terms
differ by a function whose absolute integral is $O_q(1)$.

On $\theta_0\le|\theta|\le\pi$, the formula in terms of
$S_x,S_y,S_{xy}$ from Lemma~\ref{lem:covariance} gives an $O_q(1)$ bound
for the self-, cross-, and disk covariance integrands, including at zeros
of the cross-kernel.

\medskip\noindent\textit{The logarithmic term.}
The bounded corrections to $L_{xy}$ contribute $O_q(1)$ to the integral,
since $\log(1-S_F)$ and $\log(1-S_c)$ have uniformly bounded absolute
integrals. For the self-correlation, the preceding relative-gap estimate gives
\[
 \log(1-S_F)-\log(1-S_0)=O_q(\delta^2).
\]
Integrating this difference against $|\ell_{xy}|$ gives $O_q(\delta)$
because $\int|\ell_{xy}|\,d\theta=O(\delta^{-1})$.
Together with the bounds for the rational terms, this proves
$\Var W_F(r)-\Var M_+(r)=O_q(1)$.

For the cross-correlation, \eqref{eq:crossstructure} gives, uniformly
for $0<|\theta|\le\theta_0$,
\begin{equation*}
 \log\frac{1-S_c}{1-S_0}
 =\log\bigl(1+(1-\kappa)Q_0\bigr)+O_q(\delta^2).
\end{equation*}
Indeed,
\[
 \frac{1-S_c}{1-\kappa S_0}
 =1+\frac{\kappa S_0(1-e^{e_c})}{1-\kappa S_0}
 =1+O_q(\delta^2),
\]
using $S_0\theta^2/(1-\kappa S_0)\le C\delta^2$.
Also
\[
 \frac{1-\kappa S_0}{1-S_0}=1+(1-\kappa)Q_0.
\]
Substituting this identity into the logarithmic term and combining it
with the rational-term estimates yields
\begin{equation}\label{eq:covdifference}
 \Cov(W_F(r),M_+(r))-\Var M_+(r)
 =-\frac{a^2}{2\pi}\int_{-\theta_0}^{\theta_0}
 \log\left(1+\frac{(1-\kappa)\delta^2}{2a(1-\cos\theta)}\right)
 \ell_{xy}\,d\theta+O_q(1).
\end{equation}

In \eqref{eq:covdifference}, the factor $(1-\kappa)\delta^2$ has
order $\delta^3$ by \eqref{eq:kappa}. The leading contribution
therefore comes from $\theta=\delta^{3/2}v$.
For each fixed $v\ne0$,
\[
 \delta^2\ell_{xy}(\delta^{3/2}v)\longrightarrow\frac12,
\]
\[
 \log\left(1+
 \frac{(1-\kappa)\delta^2}
      {2a(1-\cos(\delta^{3/2}v))}\right)
 \longrightarrow\log(1+\beta_q/v^2).
\]
Throughout $0<|\delta^{3/2}v|\le\theta_0$, the product of the two
rescaled factors is bounded in absolute value by
\[
 C_q\log(1+C_q/v^2),
\]
an integrable function of $v$.
The bound follows from \eqref{eq:kappa},
$1-\cos\theta\ge c\theta^2$, and
$\delta^2|\ell_{xy}|\le C$.
Extend the rescaled integrand by zero outside this angular range.
Dominated convergence then applies to \eqref{eq:covdifference}.
Since
\[
 \int_{\mathbb R}\log(1+\beta_q/v^2)\,dv=2\pi\sqrt{\beta_q},
\]
equation \eqref{eq:covdifference} yields
\[
 \sqrt\delta\bigl(\Cov(W_F(r),M_+(r))-\Var M_+(r)\bigr)
 \longrightarrow-\frac{\sqrt{\beta_q}}2.
\]
To evaluate the integral, differentiate $\int_{\mathbb R}\log(1+b/v^2)\,dv$
with respect to $b>0$. Its derivative is
$\int_{\mathbb R}(v^2+b)^{-1}\,dv=\pi/\sqrt b$.
For $0<b\le1$, the integrable majorant $\log(1+v^{-2})$ shows that
the original integral tends to zero as $b\downarrow0$, which determines
the integration constant.

Finally,
\[
 \Var(W_F-M_+)=\Var W_F-\Var M_+
 -2\bigl(\Cov(W_F,M_+)-\Var M_+\bigr),
\]
which proves Theorem~\ref{thm:sharp}.
\end{proof}

\subsection{Quantitative bounds for the standardized counts}

The following bounds quantify the convergence in Theorem~\ref{thm:joint}.
For random vectors $X,Y$ in $\mathbb R^d$ with finite first moments,
define the first Wasserstein distance by
\[
 d_{W,d}(X,Y)
 =\sup_{\operatorname{Lip}(h)\le1}|\E h(X)-\E h(Y)|.
\]
The Lipschitz constant is taken with respect to Euclidean distance.
For real random variables, the Kolmogorov distance is
\[
 d_K(X,Y)=\sup_{x\in\mathbb R}|\Pr(X\le x)-\Pr(Y\le x)|.
\]
The bounded-Lipschitz distance $d_{\mathrm{BL}}$ is defined by the
supremum for $d_{W,d}$, with the additional condition $\|h\|_\infty\le1$.
Let $\mathcal R$ be the class of rectangles
$(a_1,b_1]\times(a_2,b_2]$, with infinite endpoints allowed, and set
\[
 d_{\mathcal R}(X,Y)
 =\sup_{R\in\mathcal R}|\Pr(X\in R)-\Pr(Y\in R)|.
\]

\begin{proposition}\label{prop:rates}
Fix $q\in(0,1)$ and let $q<s<\sqrt q<t<1$.
Write $N_-=N_-(s)$ and $N_+=N_+(t)$ for the counts in
Theorem~\ref{thm:joint}, and set
\[
 r_-=q/s,\quad r_+=t,\quad\eta_i=1-r_i^2,\qquad
 Z_i=\frac{N_i-\E N_i}{\sqrt{\Var N_i}},\quad i\in\{-,+\}.
\]
The denominators are positive for sufficiently small $\eta_-,\eta_+$.
If $G^{(2)}\sim\mathcal N(0,I_2)$, then, as $\eta_-,\eta_+\downarrow0$
with no restriction on their relative rates,
\[
 \max\bigl\{d_{W,2}((Z_-,Z_+),G^{(2)}),
 d_{\mathrm{BL}}((Z_-,Z_+),G^{(2)})\bigr\}
 =O_q(\eta_-^{1/4}+\eta_+^{1/4}),
\]
\[
 d_{\mathcal R}((Z_-,Z_+),G^{(2)})
 =O_q(\eta_-^{1/6}+\eta_+^{1/6}).
\]
For each fixed $u,v\in\mathbb R$, one also has
\[
 \left|\E e^{i(uZ_-+vZ_+)}-e^{-(u^2+v^2)/2}\right|
 =O_{q,u,v}(\eta_-^{1/4}+\eta_+^{1/4}).
\]
\end{proposition}

\begin{proof}
Put $H_+=F$, $H_-=\widetilde F$, $c=\sqrt q$, and
\[
 D_i=W_{H_i}(r_i)-M_i(r_i),\quad C_i=W_{H_i}(c),\quad
 V_i=v(r_i),\quad \sigma_i^2=\Var N_i,\quad
 Y_i=\frac{M_i(r_i)-m(r_i)}{\sqrt{V_i}}.
\]
Write $Z=(Z_-,Z_+)$ and $Y=(Y_-,Y_+)$.
The winding identities in the proof of Theorem~\ref{thm:joint} give
\[
 N_i-\E N_i
 =M_i(r_i)-m(r_i)+(D_i-\E D_i)-(C_i-\E C_i).
\]
Theorem~\ref{thm:sharp}, applied also to the inverted field, gives
$\Var D_i=O_q(\eta_i^{-1/2})$,
and $C_i$ has finite second moment by Lemma~\ref{lem:covariance},
since its circle is fixed.
Proposition~\ref{prop:finite} gives
\[
 \sigma_i^2=V_i+O_q(1),\qquad
 V_i=\frac{1-\eta_i}{\eta_i(2-\eta_i)}\asymp\eta_i^{-1}.
\]
Thus $\sigma_i^{-1}=O_q(\eta_i^{1/2})$ and
$\sqrt{V_i}/\sigma_i-1=O_q(\eta_i)$.
The $L^2$ triangle inequality and $\|Y_i\|_2=1$ now give
\begin{equation}\label{eq:rate-transfer}
 \|Z_i-Y_i\|_2
 \le\frac{\sqrt{\Var D_i}+\sqrt{\Var C_i}}{\sigma_i}
 +\left|\frac{\sqrt{V_i}}{\sigma_i}-1\right|
 =O_q(\eta_i^{1/4}).
\end{equation}

\emph{The scalar Stein equation.}
For a real-valued $1$-Lipschitz function $h$ and a random variable
$G\sim\mathcal N(0,1)$, there is a bounded $C^1$ function $f$ satisfying
\begin{equation}\label{eq:rate-stein}
 f'-xf=h-\E h(G),\qquad |f|\le1,\qquad \operatorname{Lip}(f')\le2.
\end{equation}
For background on Stein's equation, see \cite[Chapter~2]{CGS}.
Let $\phi,\Phi$ be the standard normal density and distribution function,
$Q=1-\Phi$, and define
\[
 A(x)=\int_{-\infty}^xh'(t)\Phi(t)dt,\quad
 B(x)=\int_x^\infty h'(t)Q(t)dt,\quad
 f(x)=-\frac{Q(x)A(x)+\Phi(x)B(x)}{\phi(x)}.
\]
Since $|h'|\le1$ almost everywhere, both integrals are finite.
Fubini gives $h(x)-\E h(G)=A(x)-B(x)$.
When the formula for $f$ is differentiated, the terms containing $h'$
cancel. The remaining expression gives Stein's equation and shows that
$f$ is continuously differentiable.
Write
\[
 A_0=x\Phi+\phi=\int_{-\infty}^x\Phi(t)dt,\qquad
 B_0=\phi-xQ=\int_x^\infty Q(t)dt.
\]
Then $|A|\le A_0$, $|B|\le B_0$ and
$(QA_0+\Phi B_0)/\phi=1$, proving $|f|\le1$.
Its derivative is locally absolutely continuous, and almost everywhere
\[
 f''=h'-\alpha A-\beta B,\qquad
 \alpha=(1+x^2)Q/\phi-x,\quad
 \beta=(1+x^2)\Phi/\phi+x.
\]
For $x>0$, integration by parts gives
\[
 Q(x)=\frac{\phi(x)}x-\int_x^\infty\frac{\phi(t)}{t^2}\,dt
 \ge\frac{\phi(x)}x-\frac{Q(x)}{x^2},\qquad
 \frac{Q(x)}{\phi(x)}\ge\frac{x}{1+x^2}.
\]
This proves $\alpha(x)\ge0$ for $x>0$; for $x\le0$ the inequality
follows directly from its formula. Since $\beta(x)=\alpha(-x)$,
both coefficients are nonnegative.
Substituting the formulas for $A_0$ and $B_0$ gives
$\alpha A_0+\beta B_0=1$, and hence $|f''|\le2$ almost everywhere.
Local absolute continuity implies the global Lipschitz bound for $f'$.
Stein's equation and the linear growth of $h$ imply that $f'$
has at most linear growth.
Both sides of Stein's equation are therefore integrable under any
probability law with finite first moment.

Let $X_1,\ldots,X_n$ be independent centered random variables
with total variance one and finite third moments,
and put $S=\sum X_j$, $S_j=S-X_j$, $a_j=\E X_j^2$.
The $2$-Lipschitz bound for the derivative in \eqref{eq:rate-stein} gives
\[
 \E Sf(S)=\sum_j a_j\E f'(S_j)+R,\qquad
 |R|\le\sum_j\E|X_j|^3.
\]
Indeed, expand $f(S_j+X_j)$ about $S_j$. Independence and
$\E X_j=0$ eliminate the constant term after multiplication by $X_j$
and taking expectation. The Lipschitz bound for $f'$ bounds the
Taylor remainder by $|X_j|^2$.
Also
\[
 \left|\E f'(S)-\sum_j a_j\E f'(S_j)\right|
 \le2\sum_j a_j\E|X_j|
 \le2\sum_j\E|X_j|^3,
\]
since $\E|X_j|\le\sqrt{a_j}$ and $a_j^{3/2}\le\E|X_j|^3$.
Substituting these estimates into Stein's equation and taking
the supremum over $h$ gives
\begin{equation}\label{eq:rate-independent}
 d_{W,1}(S,G)\le3\sum_j\E|X_j|^3.
\end{equation}
For the disk counting law \eqref{eq:disklaw},
the Bernoulli probabilities are $p_k=r^{2k}$ and
\[
 \E|B_k-p_k|^3
 =p_k(1-p_k)[p_k^2+(1-p_k)^2]\le p_k(1-p_k).
\]
Apply \eqref{eq:rate-independent} to each centered finite partial sum,
divided by its own standard deviation.
At fixed $r<1$, the variances of the centered partial sums increase to $v(r)$,
and the resulting standardized sums converge in $L^2$
to the standardized infinite sum.
Passing to the limit in the expectations of Lipschitz functions gives
\[
 d_{W,1}(Y_i,G)\le3V_i^{-1/2}=O(\eta_i^{1/2}).
\]

To compare the joint laws, replace the independent coordinates
$Y_-$ and $Y_+$ one at a time by independent standard normals.
Averaging over the other coordinate preserves the scalar
$1$-Lipschitz constant, so
\[
 d_{W,2}(Y,G^{(2)})
 \le d_{W,1}(Y_-,G)+d_{W,1}(Y_+,G).
\]
For every Euclidean $1$-Lipschitz test function,
the difference of its expectations at the coupled vectors $Z$ and $Y$ is at most
$\E\|Z-Y\|\le\sum_i\|Z_i-Y_i\|_1$.
Together with \eqref{eq:rate-transfer}, this proves the Wasserstein bound,
and the bounded-Lipschitz bound follows by restricting the class of test
functions.
The real and imaginary parts of $e^{i(ux_1+vx_2)}$ have Lipschitz
constant at most $\sqrt{u^2+v^2}$. Applying the Wasserstein bound
to these two functions gives the characteristic-function estimate.
For $u=v=0$, both characteristic functions equal one.

\emph{Rectangles.}
Upper and lower piecewise linear approximations to
$\mathbf1_{(-\infty,x]}$, with transition width $h>0$ and
Lipschitz constant $h^{-1}$, give
\[
 d_K(S,G)\le h/\sqrt{2\pi}+d_{W,1}(S,G)/h.
\]
Optimizing over $h>0$ and using the scalar Wasserstein bound gives
$d_K(Y_i,G)=O(\eta_i^{1/4})$.
Independence and a telescoping expansion of the product of the two
interval probabilities give
$d_{\mathcal R}(Y,G^{(2)})\le2\sum_i d_K(Y_i,G)$.
For arbitrary $h_i>0$, enlarge or shrink each coordinate interval
of a rectangle by $h_i$.
On the event $|Z_i-Y_i|\le h_i$ for both coordinates, $Y$ lies in
the shrunken rectangle only if $Z$ lies in the original rectangle,
and $Z$ lies in the original rectangle only if $Y$ lies in the
enlarged rectangle.
The Gaussian marginal densities are bounded, so the union bound
on the boundary strips shows that enlargement or shrinkage changes
the Gaussian probability by at most $C\sum_i h_i$.
The same bound holds if a shrunken interval is empty.
Comparing the disk vector with the Gaussian vector on these rectangles
and applying Chebyshev's inequality to the coupling error gives
\[
 d_{\mathcal R}(Z,G^{(2)})
 \le d_{\mathcal R}(Y,G^{(2)})+C\sum_i h_i
 +\sum_i\frac{\E|Z_i-Y_i|^2}{h_i^2}.
\]
Set $h_i=\eta_i^{1/6}$ and use \eqref{eq:rate-transfer}.
The last two terms are $O_q(\sum_i\eta_i^{1/6})$,
and the disk rectangle error $O(\sum_i\eta_i^{1/4})$ is of smaller order.
This proves the rectangle-distance bound.
\end{proof}

\section{Thinning annuli}\label{thin:sec:counts}

We prove Theorem~\ref{thin:thm:bands} as $h\downarrow0$,
using the transverse-band notation from Section~\ref{sec:introduction}.

\subsection{Periodization and winding covariances}

In strip coordinates \(u=y+ix\), \(0<y<1\), write
\(\mathcal F_h(u)=\sqrt h F_h(e^{-hu})\).
Its covariance is \(h\mathcal K_h(u+\bar v)\), where
\[
 \mathcal K_h(w)=\sum_{n\in\mathbb Z}
                    \frac{e^{-hwn}}{1+e^{-h(2n+1)}}.
\]
Set \(J(w)=\pi/[2\sin(\pi w/2)]\) on \(0<\Re w<2\).

\begin{lemma}\label{thin:lem:periodization}
For $0<\Re w<2$,
\begin{equation}
 h\mathcal K_h(w)=e^{hw/2}
       \sum_{m\in\mathbb Z}(-1)^m J(w+2\pi i m/h).
 \label{thin:eq:periodization}
\end{equation}
The series and each of its termwise derivatives converge normally.
For \(\Re w\) in a fixed compact subinterval of \((0,2)\) and
\(|\Im w|\le\pi/h\), each fixed-order derivative of the sum is bounded
by \(C e^{-\pi|\Im w|/2}\), uniformly for small \(h\).
On each fixed compact subset of the strip, the sum converges
exponentially fast to \(J(w)\), together with all derivatives.
\end{lemma}

\begin{proof}
For \(0<\Re w<2\), put \(f_w(t)=e^{-wt}/(1+e^{-2t})\).
The function $f_w$ and all its derivatives with respect to $t$
decay exponentially and are integrable.
Integration around a rectangular contour of height \(\pi\) gives
\[
 \int_{\mathbb R}f_w(t)\,dt=\frac\pi{2\sin(\pi w/2)}.
\]
Indeed, the integrand on the upper edge is \(e^{-i\pi w}\) times
that on the lower edge, and the two edges are traversed in opposite
directions. The only pole is \(i\pi/2\), with
residue \(e^{-i\pi w/2}/2\).
The integrals over the vertical edges tend to zero by exponential decay,
giving
\((1-e^{-i\pi w})\int f_w=\pi i e^{-i\pi w/2}\).
The same proof applies to \(w+i\omega\).

Periodize \(f_w\) on the real line with period \(h\).
The Fourier coefficient at \(2\pi m/h\) is \(h^{-1}J(w+2\pi i m/h)\).
These coefficients are absolutely summable, since \(\sin\) grows
exponentially along imaginary directions.
The Fourier series therefore represents the periodized function.
Evaluating it at \(h/2\) gives
\[
 h\sum_n f_w(h(n+1/2))
     =\sum_m(-1)^m J(w+2\pi i m/h).
\]
Multiplying by \(e^{hw/2}\) gives \eqref{thin:eq:periodization}.
For each fixed derivative order \(j\),
\[
 |J^{(j)}(s+it)|\le C_j e^{-\pi|t|/2}
\]
uniformly on the stated real compact interval.
Summing these bounds over the shifted kernels proves normal convergence
and the derivative bound throughout a full period.
If \(|t|\le\pi/h\), every term with a nonzero shift has
absolute imaginary part at least
\(2\pi/h-|t|\ge|t|\).
For bounded \(t\), the terms with nonzero shifts sum to
\(O(e^{-\pi^2/h})\), and the same estimate holds after any fixed number
of derivatives.
\end{proof}

For \(y,z\in(0,1)\), put
\[
 S_h(y,z;t)=\frac{|\mathcal K_h(y+z+it)|^2}
                 {\mathcal K_h(2y)\mathcal K_h(2z)},\qquad
 S(y,z;t)=\frac{\sin(\pi y)\sin(\pi z)}
              {\sin^2(\pi(y+z)/2)+\sinh^2(\pi t/2)},
\]
\begin{equation}
 L(y,z)=\frac1{8\pi}\int_{\mathbb R}
                         \partial_y\partial_z\Li(S(y,z;t))\,dt.
 \label{thin:eq:winding-kernel}
\end{equation}
The prefactor \(e^{hw/2}\) in \eqref{thin:eq:periodization}
cancels exactly in \(S_h\).

\begin{lemma}\label{thin:lem:winding}
The function $L$ defined by \eqref{thin:eq:winding-kernel} is finite
and continuous on $(0,1)^2$, including on the diagonal.
If \(W_h(y)=W_{F_h}(e^{-hy})\), then locally uniformly on $(0,1)^2$,
\begin{equation}
 h\Cov(W_h(y),W_h(z))\longrightarrow L(y,z).
 \label{thin:eq:winding-covariance}
\end{equation}
The mean satisfies, locally uniformly for \(y\in(0,1)\),
\begin{equation}
 \E W_h(y)=-\frac12+\frac\pi{2h}\cot(\pi y)
                              +O(h^{-1}e^{-\pi^2/h}).
 \label{thin:eq:winding-mean}
\end{equation}
\end{lemma}

\begin{proof}
Lemma~\ref{lem:logcov} gives the covariance \(\Li(S)/4\) of the log moduli
of centered jointly proper Gaussian variables.
The Jensen secant bounds in Lemma~\ref{lem:covariance} justify radial
differentiation of the circular log-modulus averages in \(L^2\).
Apply these identities to the Laurent field for each fixed $h>0$.
The angular period is \(2\pi/h\) in the \(x\) coordinate, and
the winding operator on the circular logarithmic average
is \(-h^{-1}\partial_y\).
Stationarity in \(x\) therefore gives
\begin{equation*}
 h\Cov(W_h(y),W_h(z))
   =\frac1{8\pi}\int_{-\pi/h}^{\pi/h}
                   \partial_y\partial_z\Li(S_h(y,z;t))\,dt.
\end{equation*}

For \(y,z\) in a fixed compact subinterval of $(0,1)$,
Lemma~\ref{thin:lem:periodization} bounds
\(S_h,S_{h,y},S_{h,z},S_{h,yz}\) by \(Ce^{-\pi|t|}\) when
\(|t|\) is sufficiently large within the principal period.
The first two derivatives of \(\Li\) are bounded for
\(S_h\le1/2\), including at zero.
It follows that \(\partial_y\partial_z\Li(S_h)\) is bounded by
\(Ce^{-\pi|t|}\) in this region.
For bounded \(t\) away from the coincidence set \(t=0,y=z\), the
limiting correlation is bounded away from one. Normal convergence then
gives a uniform bound on the mixed derivative.

Near $y=z$ and $t=0$, we estimate the correlation gap.
Let \(d=y-z\).
The normalized correlation kernels are even in \(d\) and \(t\) when written
with midpoint \((y+z)/2\). At \(d=t=0\) they equal one.
The displayed formula for $S$ shows that the Hessian of \(1-S\)
there is positive definite, uniformly when the midpoint ranges over a
compact subinterval of $(0,1)$. Uniform normal convergence through
sufficiently many derivatives therefore gives, for all small \(h\),
\begin{equation*}
 1-S_h\asymp d^2+t^2,\qquad
 |\partial_y\log S_h|+|\partial_z\log S_h|
                  \le C(|d|+t^2),\qquad
 |\partial_{yz}\log S_h|\le C
\end{equation*}
in a fixed neighborhood.
The first radial derivatives vanish at $y=z$ and $t=0$. Their dependence
on $t$ is even, so their variation from $t=0$ is quadratic.
These statements also hold for \(S\).
With \(\ell=\log S_h\), differentiation gives
\[
 \partial_{yz}\Li(S_h)
  =\frac{S_h}{1-S_h}\ell_y\ell_z
                         -\log(1-S_h)\ell_{yz}.
\]
Using \((d^2+t^4)/(d^2+t^2)\le1+t^2\), we bound the mixed derivative
in absolute value by \(C(1+|\log|t||)\) in that neighborhood.
This bound is integrable in \(t\), uniformly in \(y,z,h\).
The value at the coincidence point does not affect the integral.
The same bounds for the unmixed first and second radial derivatives
show that the integrated logarithmic covariance of the strip field
is locally twice continuously differentiable.
Dominated convergence proves \eqref{thin:eq:winding-covariance},
uniformly on every compact square in the open domain,
and continuity of \eqref{thin:eq:winding-kernel}.

For the mean, averaging the Gaussian log modulus gives
\(\E W_h(y)=-(2h)^{-1}\partial_y\log\mathcal K_h(2y)\).
Equation \eqref{thin:eq:periodization} at real arguments gives
\[
 \log\mathcal K_h(2y)
 =-\log h+hy+\log J(2y)+O(e^{-\pi^2/h}),
\]
with the same error estimate after differentiation.
Since \(\partial_y\log J(2y)=-\pi\cot(\pi y)\),
this proves \eqref{thin:eq:winding-mean}.
\end{proof}

\subsection{The strip flux covariance}

Let $\mathcal F$ be the centered proper Gaussian strip field
with covariance $J(u+\bar v)$.
Peres and Vir\'ag~\cite[Theorem~1]{PV} proved that the zeros of the
ordinary disk GAF have determinantal correlations
with kernel $\pi^{-1}(1-z\bar w)^{-2}$ relative to Lebesgue area.

\begin{lemma}\label{thin:lem:flux}
Let $p_k(y)=e^{-ky}/\sqrt{D(k)}$.
For every real $f\in C_c^\infty((0,1))$,
\begin{equation}\label{thin:eq:smooth-flux}
 \iint f'(y)f'(z)L(y,z)\,dy\,dz
  =\int_{\mathbb R}
    \left\{\int_0^1 f(y)^2p_k(y)^2\,dy
       -\left(\int_0^1 f(y)p_k(y)^2\,dy\right)^2\right\}\,dk.
\end{equation}
If $g$ is a real compactly supported finite step function on $(0,1)$ with
$g'=\sum_j\gamma_j\delta_{y_j}$, then
\begin{equation}\label{thin:eq:step-flux}
 \sum_{j,k}\gamma_j\gamma_k L(y_j,y_k)
                 =\int_{\mathbb R}\operatorname{Var}_{\nu_k}(g)\,dk,
 \qquad d\nu_k(y)=p_k(y)^2dy.
\end{equation}
\end{lemma}

\begin{proof}
Write \(\zeta=e^{i\pi u}\), which maps the vertical strip conformally
onto the upper half-plane, and \(\phi(\zeta)=(\zeta-i)/(\zeta+i)\).
The disk GAF, composed with \(\phi\) and multiplied by
\(\sqrt{2\pi}e^{i\pi u/2}/(e^{i\pi u}+i)\), has covariance
\(J(u+\bar v)\), by direct substitution.
The multiplier is analytic and nonvanishing.
Since both fields are centered and proper, equality of their
covariance kernels identifies their laws.
Transforming the disk zero correlations by the conformal Jacobians
shows that the strip zeros form a determinantal point process (DPP),
with kernel relative to Lebesgue area
\begin{equation}
 B(u,v)=\frac\pi{4\sin^2(\pi(u+\bar v)/2)}.
 \label{thin:eq:strip-bergman}
\end{equation}
In particular, its one-point density is
\(\rho(y)=\pi/[4\sin^2(\pi y)]\).

With \(p_k(y)=e^{-ky}/\sqrt{D(k)}\), the strip kernel has the spectral
representation
\begin{equation}
 B(y+ix,z+ix')=\frac1{2\pi}\int_{\mathbb R}
                  p_k(y)p_k(z)e^{-ik(x-x')}\,dk.
 \label{thin:eq:strip-spectral}
\end{equation}
To verify it, expand \((1-e^{-2k})^{-1}\) for \(k>0\),
treat \(k<0\) by reflection, and integrate \(ke^{-kw}\).
The result is \(\pi^{-1}\sum_{m\in\mathbb Z}(w+2m)^{-2}\),
where \(w=y+z+i(x-x')\).
To identify this sum with \eqref{thin:eq:strip-bergman}, consider the
difference between
\(\sum_{m\in\mathbb Z}(w+2m)^{-2}\) and
\(\pi^2/[4\sin^2(\pi w/2)]\).
This difference is entire and $2$-periodic, since the double poles cancel.
Both terms tend to zero at imaginary infinity:
the series is bounded by \(C/|\Im w|\) on a fundamental strip.
The difference is bounded on the remaining compact part of that strip;
periodicity therefore makes it bounded on the whole plane.
Liouville's theorem then shows that the difference is zero.
The expansions are absolutely integrable for \(0<\Re w<2\).
Write $B(y,z;t)=B(y+it,z)$.
The spectral factors in \eqref{thin:eq:strip-spectral} are exponentially
integrable, so Plancherel's identity gives
\begin{equation}
 \int_{\mathbb R}|B(y,z;t)|^2\,dt
         =\frac1{2\pi}\int_{\mathbb R}p_k(y)^2p_k(z)^2\,dk.
 \label{thin:eq:strip-plancherel}
\end{equation}

We compute the limiting variance per unit length of a smooth zero
statistic in two ways: from the DPP kernel and from integrated log moduli.
For the log-modulus calculation, we integrate the log modulus in $L^2$
and express the variance limit in terms of $L$.
Fix a real \(g\in C_c^\infty((0,1))\) and a nonzero real
\(\chi\in C_c^\infty(\mathbb R)\).
Set \(\chi_R(x)=\chi(x/R)\), and let \(Z_R\)
be the strip zero linear statistic with weight \(g(y)\chi_R(x)\).
The DPP one- and two-point formulas and
\eqref{thin:eq:strip-plancherel} give
\begin{equation}
 \lim_{R\to\infty}\frac{\Var Z_R}{R\int\chi^2}
  =\frac1{2\pi}\int_{\mathbb R}
       \left\{\int_0^1g(y)^2p_k(y)^2\,dy
             -\left(\int_0^1g(y)p_k(y)^2\,dy\right)^2\right\}dk.
 \label{thin:eq:dpp-flux}
\end{equation}
The longitudinal autocorrelation of $\chi_R$, divided by \(R\), tends
to \(\int\chi^2\) and is bounded in absolute value by the same quantity.
The strip kernel decays exponentially in \(t\), uniformly on transverse
compact sets, so dominated convergence gives the limit.

The distributional argument principle
\(\Delta\log|\mathcal F|=2\pi Z_{\mathcal F}\) gives
\[
 Z_R=\frac1{2\pi}\int\log|\mathcal F(y+ix)|
              \{g''(y)\chi_R(x)+g(y)\chi_R''(x)\}\,dy\,dx.
\]
The centered log modulus of the strip field has covariance
\(C(y,z;t)=\Li(S(y,z;t))/4\).
This covariance is absolutely integrable in \(t\), uniformly on
transverse compact sets.
Applying Cauchy--Schwarz to the longitudinal convolution with
derivatives of orders \(i,j\in\{0,2\}\) gives the bound
\(C R^{1-i-j}\), since
\(\|\chi_R^{(i)}\|_2=R^{1/2-i}\|\chi^{(i)}\|_2\).
Thus the variance contributions involving \(\chi_R''\)
vanish after division by \(R\).
For the term containing two factors of $\chi_R$, the autocorrelation
limit gives
\begin{equation}
 \lim\frac{\Var Z_R}{R\int\chi^2}
  =\frac1{4\pi^2}\iint g''(y)g''(z)
                      \left(\int_{\mathbb R}C(y,z;t)\,dt\right)dy\,dz
  =\frac1{2\pi}\iint g'(y)g'(z)L(y,z)\,dy\,dz.
 \label{thin:eq:log-flux}
\end{equation}
The logarithmic integrals lie in \(L^2\) on their compact supports.
By Lemma~\ref{thin:lem:winding}, the integrated covariance is twice
continuously differentiable.
Integration by parts gives the last equality,
with no boundary terms because the test functions are compactly supported.

Equating \eqref{thin:eq:dpp-flux} and \eqref{thin:eq:log-flux}
proves \eqref{thin:eq:smooth-flux}.
For a compactly supported finite step function $g$,
take smooth convolutions $g_\epsilon$ whose supports lie
in one compact transverse interval $K\subset(0,1)$.
They are uniformly bounded, converge to $g$ almost everywhere,
and their derivative measures converge to $g'$
with uniformly bounded total variation.
By continuity of $L$, the left side of \eqref{thin:eq:smooth-flux}
converges to the endpoint sum in \eqref{thin:eq:step-flux}.
Since $\int_0^1p_k^2=1$, Cauchy--Schwarz bounds the squared mean,
and hence each spectral integrand, by $C\int_Kp_k(y)^2dy$.
This bound decays exponentially as $k\to\pm\infty$, uniformly over
the chosen supports.
Dominated convergence therefore proves \eqref{thin:eq:step-flux}.
\end{proof}

\subsection{The long-circle Gaussian field and exact projections}

Put $T=2\pi/h$ and $u=y+ix$. To remove the exponential factor
from the covariance, define
\[
 Y_h(u)=\sqrt h\,e^{-hu/2}F_h(e^{-hu}).
\]
Lemma~\ref{thin:lem:periodization} gives the exact covariance
\begin{equation}
 A_h(u+\bar v),\qquad
 A_h(w)=\sum_{\nu\in\mathbb Z}(-1)^\nu
           J(w+2\pi i\nu/h),\qquad
 J(w)=\frac\pi{2\sin(\pi w/2)}.
 \label{thin:eq:gauged-covariance}
\end{equation}
The field is antiperiodic in $x$ with period $T$;
its log modulus and every even Gaussian polynomial used below are periodic.
Write
\[
 Z_{h,y}(x)=Y_h(y+ix)/\sqrt{A_h(2y)}.
\]
Each $Z_{h,y}(x)$ is standard proper complex Gaussian.
For fixed levels in a transverse compact interval,
the covariance and all fixed-order radial derivatives satisfy
\begin{equation}
 |A_h^{(j)}(y+z+it)|\le C_j e^{-\pi d_T(t)/2},\qquad
 d_T(t)=\min_{\nu\in\mathbb Z}|t-\nu T|.
 \label{thin:eq:gaussian-decay}
\end{equation}
We obtain this bound on $[-T/2,T/2]$ and extend it by antiperiodicity.
The diagonal normalizations are bounded above and away from zero,
and their derivatives of each fixed order are bounded,
uniformly for small $h$.
On bounded intervals in $t$, the covariance kernels converge
with all derivatives to $J(y+z+it)$.

Write $\mathsf L_m=L_m$ for the Laguerre polynomials used in the proof
of Lemma~\ref{lem:logcov}.
For a standard proper complex Gaussian variable $Z$,
\begin{equation}
 \log|Z|-\E\log|Z|
 =-\frac12\sum_{m\ge1}\frac{\mathsf L_m(|Z|^2)}m
                       \quad\hbox{in }L^2.
 \label{thin:eq:log-projection}
\end{equation}
For a jointly proper standard pair with correlation $\rho$,
the same proof gives
\begin{equation}
 \E \mathsf L_m(|Z_1|^2)\mathsf L_n(|Z_2|^2)
       =\mathbf1_{m=n}|\rho|^{2m}.
 \label{thin:eq:laguerre-covariance}
\end{equation}

The terms in \eqref{thin:eq:log-projection} are the orthogonal
Gaussian chaos projections.
The Wick generating exponential $\exp(sZ+t\bar Z-st)$ shows
$\mathsf L_m(|Z|^2)=(-1)^m:Z^m\bar Z^m:/m!$,
which belongs to the chaos of degree $2m$.
For Gaussian linear forms in the countable coefficient sequence,
the same identity follows by approximation with finitely many
coefficients in $L^p$, for each finite $p$.
Wick polynomials of different degrees are orthogonal.
The exponential-transform argument also shows that Gaussian polynomials
are dense in each finite-dimensional Gaussian $L^2$ space.
Conditional expectations with respect to finitely many coordinates
then approximate every $L^2$ function of the countable Gaussian
coefficient sequence, proving completeness of the chaos decomposition.

By Lemma~\ref{lem:covariance}, the circular logarithmic average defining
$W_h(y)$ is $L^2$ differentiable for each fixed $h$ and $y$.
Bounded orthogonal projections commute with this $L^2$ derivative.
Integrating \eqref{thin:eq:log-projection} around the circle and
differentiating the orthogonal projections therefore gives the exact series
\[
 \sqrt h\{W_h(y)-\E W_h(y)\}
       =\sum_{m\ge1} X_{m,h}(y)\quad\hbox{in }L^2,
\]
\begin{equation}
 X_{m,h}(y)=\frac1{\sqrt{2\pi T}}\int_0^T P_{m,h,y}(x)\,dx,
 \qquad
 P_{m,h,y}(x)=\frac1{2m}\partial_y\mathsf L_m(|Z_{h,y}(x)|^2).
 \label{thin:eq:winding-projection}
\end{equation}
Expressing the winding through the transformed field gives
\[
 W_h(y)
 =-\frac1{hT}\partial_y\int_0^T\log|Y_h(y+ix)|\,dx-\tfrac12.
\]
Thus the analytic factor changes the winding by a deterministic term,
which disappears after centering.
The normalization in \eqref{thin:eq:winding-projection}
follows from the winding operator $-h^{-1}\partial_y$
on the circular logarithmic average and the identity $hT=2\pi$.
The integrand $P_{m,h,y}$ is a polynomial in $Z_{h,y}$
and its radial derivative and has moments of every order.
It is the limit in $L^2$ of difference quotients of chaos polynomials
of degree $2m$ and therefore remains in that chaos.
Thus the differentiated projections are mutually orthogonal in $L^2$.
Stationarity and \eqref{thin:eq:laguerre-covariance},
with $S_h$ the squared normalized correlation, give
\begin{equation}
 \Cov(X_{m,h}(y),X_{n,h}(z))
 =\mathbf1_{m=n}\frac1{8\pi m^2}
       \int_{-T/2}^{T/2}\partial_y\partial_z S_h(y,z;t)^m\,dt.
 \label{thin:eq:projection-covariance}
\end{equation}

\subsection{Uniform chaos-tail bounds}

\begin{lemma}\label{thin:lem:chaos-tail}
Let $K\subset(0,1)$ be a compact interval. There are constants
$h_0>0$ and $C<\infty$ such that, for $0<h<h_0$, $y\in K$, and every
integer $m\ge1$,
\begin{equation}
 \Var X_{m,h}(y)\le C m^{-3/2}.
 \label{thin:eq:degree-bound}
\end{equation}
Moreover, for every integer $M\ge1$, every finite set of levels
$y_j\in K$, and all real coefficients $\gamma_j$, the same constants give
\begin{equation}
 \E\left|\sum_j\gamma_j\sum_{m>M}X_{m,h}(y_j)\right|^2
       \le C\left(\sum_j|\gamma_j|\right)^2 M^{-1/2}.
 \label{thin:eq:chaos-tail}
\end{equation}
\end{lemma}

\begin{proof}
We differentiate in the two radial variables before setting them equal.
At $y=z$ and $|t|\le t_0$, the normalized self-correlation estimates
in Lemma~\ref{thin:lem:winding} give, uniformly,
\[
 1-S_h\ge ct^2,\quad
 |S_{h,y}|+|S_{h,z}|\le Ct^2,\quad |S_{h,yz}|\le C.
\]
The first radial derivatives vanish at coincidence and are even in $t$,
which gives the second estimate.
Shrinking $t_0$ if necessary gives $S_h\le e^{-ct^2}$.
For $m\ge3$, differentiating the power therefore yields
\[
 \frac{|\partial_{yz}S_h^m|}{m^2}
 \le C\left\{m^{-1}e^{-c(m-1)t^2}
                 +t^4e^{-c(m-2)t^2}\right\}.
\]
Integration gives $Cm^{-3/2}+Cm^{-5/2}$.
Outside this neighborhood, within the principal period,
the uniform bound $S_h\le\kappa<1$ follows from
uniform convergence on $t_0\le|t|\le R$
and \eqref{thin:eq:gaussian-decay} for large $|t|$.
There we also have
$|S_{h,y}|,|S_{h,z}|,|S_{h,yz}|\le Ce^{-\pi|t|}$.
The integral over this region is bounded by
$C\{m^{-1}\kappa^{m-1}+\kappa^{m-2}\}\le Cm^{-3/2}$.
The derivative formula also applies where the correlation vanishes.
The cases $m=1,2$ follow separately from the same derivative estimates
near and away from coincidence.
Equation~\eqref{thin:eq:projection-covariance}
now gives \eqref{thin:eq:degree-bound}.
Orthogonality across $m$ and Cauchy--Schwarz at each $m$
give \eqref{thin:eq:chaos-tail}, since
$\sum_{m>M}m^{-3/2}\le CM^{-1/2}$.
\end{proof}

For each fixed integer $M\ge1$, the covariance matrix of the vector
truncated at $m=M$ in \eqref{thin:eq:winding-projection} converges.
Indeed, \eqref{thin:eq:gaussian-decay} and the bounds near coincidence
allow us to apply dominated convergence.
Let $v_M$ be the limiting variance of its linear combination
with coefficients $\gamma_j$.
Lemma~\ref{thin:lem:winding} and \eqref{thin:eq:chaos-tail} imply
\begin{equation}
 0\le v-v_M\le C_\gamma M^{-1/2},
 \label{thin:eq:variance-tail}
\end{equation}
where $v$ is the limiting variance of the full scalar combination.
For each $h$, the truncated sum is orthogonal to the remaining chaos
degrees. Passing to the limit in the variance tail estimate gives the bound.

\subsection{Connected diagrams for the finite projection}

Fix a finite chaos cutoff. We prove a Gaussian limit for the truncated
sum by estimating its cumulants as $h\downarrow0$.
For each fixed longitudinal coordinate $x$, form the proper complex
Gaussian vector consisting of $Y_h(y_j+ix)$ and its radial derivative
at the distinct levels $y_j$.
Its covariance matrix has a strictly positive definite limit.
This follows from the positive spectral representation of the strip kernel,
\[
 J(w)=\int_{\mathbb R}\frac{e^{-wk}}{1+e^{-2k}}\,dk.
\]
The Gram functions for values and derivatives are
$e^{-y_jk}$ and $-ke^{-y_jk}$.
A linear combination of these exponential polynomials that vanishes
almost everywhere must vanish identically.
For the smallest level $y_j$, multiply the vanishing combination by
$e^{y_jk}$ and let $k\to+\infty$. Its polynomial coefficient must
vanish. Repeating this argument proves that all coefficients vanish.
Hence the Gram functions are linearly independent.
By the normal convergence in \eqref{thin:eq:gauged-covariance},
the smallest covariance eigenvalue stays bounded away from zero
for small $h$.

Multiply this vector by the inverse square root of its covariance
matrix, which does not depend on $x$. Then multiply its real and
imaginary parts by $\sqrt2$.
This gives a real Gaussian vector $V_h(x)$ whose pointwise covariance
is the identity.
The whitening matrices and their inverses are uniformly bounded.
Equation \eqref{thin:eq:gaussian-decay}, including its first two derivatives,
therefore gives, for every pair of coordinates,
\begin{equation}
 |\E V_{h,p}(x)V_{h,q}(x')|
              \le C e^{-\alpha d_T(x-x')},\qquad \alpha>0.
 \label{thin:eq:whitened-decay}
\end{equation}
For fixed $M$ and fixed coefficients $\gamma_j$, the polynomial
\[
 Q_h(x)=\sum_j\gamma_j\sum_{m\le M}P_{m,h,y_j}(x)
\]
has a finite Wick expansion of even degrees between $2$ and $2M$,
with coefficients bounded uniformly in $h$.
The derivative in \eqref{thin:eq:winding-projection} preserves the chaos
degree, and the coefficients remain bounded under the whitening
transformation.
The polynomial $Q_h$ is centered and has moments of every order.
Evenness makes $Q_h$ periodic despite the antiperiodicity of $V_h$.

For every fixed integer $r\ge2$, Wick's formula and the moment
partition formula express
$\kappa_r(Q_h(x_1),\ldots,Q_h(x_r))$
as a finite sum of connected pairing diagrams on these $r$ vertices.
To see why only connected diagrams contribute,
expand each Wick polynomial into its labelled Gaussian slots.
Because the polynomial at each vertex is Wick ordered, slots
are paired only across different vertices.
In the cumulant partition sum, a diagram whose connected components
form a partition with $j$ blocks receives coefficient
$\sum_{\pi\ \text{coarser}}(-1)^{|\pi|-1}(|\pi|-1)!$.
This is zero for $j>1$ and one for $j=1$,
as is seen by comparing coefficients of $\log\exp(z)=z$
in its exponential series.
Thus only connected pairing diagrams remain.
The diagram expansion uses only finitely many moments of the Gaussian
polynomials.

Every connected diagram contains a spanning tree.
Bound its tree edges with \eqref{thin:eq:whitened-decay}
and every other covariance edge by a constant.
Integrate leaves successively on the circle:
for all $T$ large,
\[
 \int_0^T e^{-\alpha d_T(x-x')}dx\le2/\alpha.
\]
The final integration, over the root variable, contributes $T$.
For fixed $r,M$, levels and coefficients, the number of diagrams and
their coefficient bounds are independent of $h$.
Consequently, for
\[
 U_{M,h}=(2\pi T)^{-1/2}\int_0^T Q_h(x)\,dx,
\]
\begin{equation*}
 |\kappa_r(U_{M,h})|\le C_{r,M,\gamma}T^{1-r/2}.
\end{equation*}
Minkowski's inequality and the moment bounds for $Q_h(x)$ show that
the integral has moments of every order.
Fubini's theorem therefore allows the finite cumulant formula
to be integrated term by term.
In particular, all cumulants of orders at least three tend to zero,
the first is zero, and the second converges to $v_M$
by \eqref{thin:eq:projection-covariance}.

The moment--cumulant formula then gives convergence of every moment
to the corresponding centered Gaussian moment.
For real $t$ and any integer $p\ge1$,
Taylor's formula on the real line gives
\[
 \left|e^{itx}-\sum_{j=0}^{2p-1}\frac{(itx)^j}{j!}\right|
                  \le\frac{|tx|^{2p}}{(2p)!}.
\]
Apply this to $U_{M,h}$ and to a centered normal $G_M$ of variance $v_M$.
At fixed $p$, the corresponding moments converge, so
\[
 \limsup_{h\downarrow0}
       |\E e^{itU_{M,h}}-e^{-t^2v_M/2}|
       \le2\frac{(t^2v_M/2)^p}{p!}\longrightarrow0
                      \quad(p\to\infty).
\]
If $v_M=0$, convergence of the second moment gives the same conclusion
directly in $L^2$.
This proves the scalar Gaussian limit for each finite chaos projection.

\subsection{Proof of the finite-band theorem}

\begin{proof}[Proof of Theorem~\ref{thin:thm:bands}]
Fix $v\in\mathbb R^r$ and put $g=\sum_i v_i\mathbf1_{I_i}$.
List its distinct endpoints as $y_1,\ldots,y_\ell$
and its jumps as $\gamma_1,\ldots,\gamma_\ell$,
so $g'=\sum_j\gamma_j\delta_{y_j}$.
The finitely many deterministic endpoint circles are almost surely zero-free.
The argument principle gives
\begin{equation*}
 \sum_i v_iN_{i,h}=\sum_j\gamma_jW_h(y_j).
\end{equation*}
Lemma~\ref{thin:lem:winding} and \eqref{thin:eq:step-flux}
show that the scaled variance converges to
\begin{equation}\label{thin:eq:general-variance}
 \sum_{j,k}\gamma_j\gamma_k L(y_j,y_k)
   =\int_{\mathbb R}\operatorname{Var}_{\nu_k}(g)\,dk
   =v^T\Sigma v.
\end{equation}
Expanding $g^2$ gives \eqref{thin:eq:band-sigma};
its first term involves interval intersections
and thus includes overlapping or coincident bands.
Polarization gives every covariance matrix entry
in \eqref{thin:eq:band-clt}.
Finiteness follows from the integrable bound on spectral mass
over compact transverse sets in the proof of Lemma~\ref{thin:lem:flux}.

To prove convergence in distribution, set
\[
 U_h=\sqrt h\sum_i v_i(N_{i,h}-\E N_{i,h}).
\]
Its projection onto chaos degrees at most $2M$ is $U_{M,h}$, with the
distinct endpoints $y_j$ and their jumps $\gamma_j$ as levels and
coefficients.
By \eqref{thin:eq:chaos-tail},
$\|U_h-U_{M,h}\|_2\le C_vM^{-1/4}$,
uniformly for small $h$.
By orthogonality, \eqref{thin:eq:general-variance} and
\eqref{thin:eq:variance-tail} give
$0\le v^T\Sigma v-v_M\le C_vM^{-1/2}$.
Using the characteristic-function limit for each finite projection and
$|e^{itx}-e^{ity}|\le |t||x-y|$,
we first let $h\downarrow0$ at fixed $M$,
and then let $M\to\infty$.
Hence
\begin{equation}\label{thin:eq:scalar-characteristic-function}
 \E e^{itU_h}\longrightarrow e^{-t^2v^T\Sigma v/2}.
\end{equation}
If $v^T\Sigma v=0$, \eqref{thin:eq:general-variance} gives
$U_h\to0$ in $L^2$.

The scaled marginal second moments are bounded, so Chebyshev's
inequality gives tightness of the count vectors.
If $G$ is a standard Gaussian vector in $\mathbb R^r$, then
$\Sigma^{1/2}G$ defines the limiting Gaussian law even when
$\Sigma$ is singular.
Equation~\eqref{thin:eq:scalar-characteristic-function}
gives convergence of the joint characteristic functions at every vector
argument.
Fourier inversion for compactly supported smooth Euclidean test functions,
whose Fourier transforms are integrable, gives convergence of expectations
by dominated convergence.
Uniform smooth approximation of a bounded continuous test function on
a large ball, together with the uniform second-moment tail bound,
proves the joint weak convergence in \eqref{thin:eq:band-clt},
including when the limiting covariance matrix is singular.

Each $\nu_k$ has a strictly positive density on $(0,1)$.
A nonzero step function $g$ of the stated form cannot be constant
almost everywhere, because it vanishes on a nonempty open interval
outside the compact union of the bands.
Thus $\operatorname{Var}_{\nu_k}(g)>0$ for every finite $k$
whenever $g\ne0$ in $L^2((0,1))$.
Conversely the variance is zero when $g=0$ almost everywhere.
For a positive semidefinite matrix,
$v^T\Sigma v=0$ is equivalent to $\Sigma v=0$,
proving the exact nullspace assertion.
In each null direction, the count combination is itself zero almost surely,
since the step function vanishes away from the deterministic
zero-free endpoint circles.
Every single band and its complement have positive $\nu_k$-mass.
Two disjoint nonempty bands have zero intersection mass
and a strictly negative product term in \eqref{thin:eq:band-sigma}.
This proves positivity of the marginal limits and negativity of the
disjoint-band covariance limits.

Finally \eqref{thin:eq:winding-mean} applies to every endpoint.
Its constant $-1/2$ cancels in each winding difference,
leaving \eqref{thin:eq:band-means}.
The remainder is negligible after multiplication by $\sqrt h$,
so the leading mean terms may replace the exact expectations.
\end{proof}

\subsection{Explicit constants and two adjacent bands}

For $0<s<2$, set
\[
 \mathcal M(s)=\frac{1-s}{4}\cot(\pi s),
 \quad 0<s<2,\ s\ne1,
 \qquad \mathcal M(1)=-\frac1{4\pi},
\]
and, for intervals \(I=(d,e)\), \(J=(f,g)\)
compactly contained in \((0,1)\),
\[
 \mathcal P(I,J)
 =\mathcal M(e+g)-\mathcal M(d+g)
  -\mathcal M(e+f)+\mathcal M(d+f),\qquad
 \mu_I=\frac{\cot(\pi d)-\cot(\pi e)}4.
\]
With this value at $s=1$, the function $\mathcal M$ is smooth on $(0,2)$.

\begin{corollary}\label{thin:cor:two-bands}
For $0<a<c=1/2<b<1$, let $A=(a,c)$ and $B=(c,b)$, and define
\[
 N_{+,h}=Z_{F_h}\{e^{-hc}<|z|<e^{-ha}\},\qquad
 N_{-,h}=Z_{F_h}\{e^{-hb}<|z|<e^{-hc}\}.
\]
Their variance and covariance limits are
\begin{equation}
 \begin{aligned}
 h\Var N_{+,h}&\longrightarrow2\pi\{\mu_A-\mathcal P(A,A)\},\\
 h\Var N_{-,h}&\longrightarrow2\pi\{\mu_B-\mathcal P(B,B)\},\\
 h\Cov(N_{+,h},N_{-,h})&\longrightarrow-2\pi\mathcal P(A,B).
 \end{aligned}
 \label{thin:eq:closed-covariances}
\end{equation}
In particular, at \(a=1/4\), \(b=3/4\),
\[
 h\Var N_{\pm,h}\longrightarrow\frac\pi4+\frac12,
 \qquad
 h\Cov(N_{+,h},N_{-,h})\longrightarrow\frac\pi4-1,
\]
\begin{equation}
 \operatorname{Corr}(N_{+,h},N_{-,h})
 \longrightarrow\frac{\pi-4}{\pi+2},\qquad
 h\Var(N_{+,h}+N_{-,h})\longrightarrow\pi-1.
 \label{thin:eq:symmetric-constants}
\end{equation}
The vector obtained by centering each count by its exact mean and
dividing by its marginal standard deviation has a Gaussian limit
with correlation strictly between $-1$ and zero.
At $a=1/4$, $b=3/4$,
\begin{equation*}
 \sqrt h
 \begin{pmatrix}
 N_{+,h}-\pi/(2h)\\
 N_{-,h}-\pi/(2h)
 \end{pmatrix}
 \Longrightarrow
 \mathcal N_2\!\left(
 0,
 \begin{pmatrix}
 \pi/4+1/2&\pi/4-1\\
 \pi/4-1&\pi/4+1/2
 \end{pmatrix}
 \right).
\end{equation*}
\end{corollary}

\begin{proof}
The strip kernel, relative to Lebesgue area, is
\(B(y+ix,z)=\pi/[4\sin^2(\pi(y+z+ix)/2)]\).
Put \(s=y+z\).
The integral of the kernel's squared modulus in the longitudinal direction is
\[
 H(s)=\frac{\pi^2}{16}\int_{\mathbb R}
 \frac{dt}{\{\sin^2(\pi s/2)+\sinh^2(\pi t/2)\}^2}.
\]
Substituting \(v=\tanh(\pi t/2)\) and writing
\(\omega=\sin^2(\pi s/2)\) gives
\[
 H(s)=\frac\pi8\int_{-1}^1
       \frac{1-v^2}{\{\omega+(1-\omega)v^2\}^2}\,dv.
\]
For \(\omega\in(0,1)\), integration gives
\[
 \mathcal I(\omega)
 =\int_{-1}^1\frac{dv}{\omega+(1-\omega)v^2}
 =\frac{2\arctan\sqrt{(1-\omega)/\omega}}
        {\sqrt{\omega(1-\omega)}}.
\]
For \(s\in(0,2)\), \(s\ne1\), this can equivalently be written
\(\mathcal I(\omega(s))=2\pi(1-s)/\sin(\pi s)\).
Differentiating the integral with respect to $\omega$ gives
\(H(s)=-\pi\mathcal I'(\omega(s))/8\).
Consequently
\begin{equation}
 H(s)=\frac\pi2
 \frac{\sin(\pi s)+\pi(1-s)\cos(\pi s)}{\sin^3(\pi s)}
       =\mathcal M''(s).
 \label{thin:eq:primitive-kernel}
\end{equation}
At \(s=1\), the original integral is continuous and equals \(\pi/6\),
by the same substitution with \(\omega=1\).
Thus \eqref{thin:eq:primitive-kernel} also holds there
by continuity and the Taylor expansion of \(\mathcal M\)
at its removable singularity.
All differentiations are taken on transverse compact sets,
away from the endpoint singularities.

Equation~\eqref{thin:eq:strip-plancherel} gives
\[
 H(y+z)=\frac1{2\pi}\int_{\mathbb R}
 \frac{e^{-2ky}e^{-2kz}}{D(k)^2}\,dk,
 \quad D(k)=\int_0^1e^{-2ky}\,dy.
\]
Tonelli's theorem therefore implies
\[
 \frac1{2\pi}\int_{\mathbb R}\lambda_I(k)\lambda_J(k)\,dk
 =\int_I\int_J\mathcal M''(y+z)\,dy\,dz
 =\mathcal P(I,J).
\]
Similarly the strip intensity
\(\rho(y)=\pi/[4\sin^2(\pi y)]\), or the diagonal spectral identity, gives
\[
 \frac1{2\pi}\int_{\mathbb R}\lambda_I(k)\,dk
 =\int_I\rho(y)\,dy=\mu_I.
\]
Substitution in \eqref{thin:eq:band-sigma}
proves \eqref{thin:eq:closed-covariances}.
At $a=1/4$, $b=3/4$,
\[
 \mu_A=\mu_B=\frac14,\qquad
 \mathcal P(A,A)=\mathcal P(B,B)=\frac18-\frac1{4\pi},\qquad
 \mathcal P(A,B)=-\frac18+\frac1{2\pi}.
\]
These values follow from
\(\mathcal M(1/2)=\mathcal M(3/2)=0\),
\(\mathcal M(3/4)=\mathcal M(5/4)=-1/16\),
and \(\mathcal M(1)=-1/(4\pi)\).
To obtain \eqref{thin:eq:symmetric-constants},
combine the entries in \eqref{thin:eq:closed-covariances}
and use the positive limiting marginal variance.
Theorem~\ref{thin:thm:bands} gives the Gaussian conclusions;
the positive limiting variances allow division by the exact
marginal standard deviations.
For general $a,b$, the two disjoint interval indicators are linearly
independent, so the limiting matrix is strictly positive definite;
its cross entry is strictly negative.
The limiting correlation therefore lies in $(-1,0)$.
\end{proof}

\section*{Statements and declarations}

\noindent\textbf{Funding.}
Xiaolin Zhu is supported by the National Natural Science Foundation of China,
Tianyuan Fund for Mathematics (Grant No.\ 12526614).

\medskip
\noindent\textbf{AI statement.}
Artificial-intelligence tools were used for language editing,
\LaTeX\ formatting, and limited assistance with local mathematical
reasoning. The authors take full responsibility for the mathematical
content, the references, and the final text.

\end{document}